\documentclass[a4paper,12pt]{amsart}
\usepackage[a4paper,margin=1.2in]{geometry}
\usepackage[T1]{fontenc}
\usepackage[utf8]{inputenc}
\usepackage[english]{babel}
\usepackage{lmodern}
\usepackage{microtype}
\usepackage{xcolor}
\usepackage{amsmath,amssymb,amsfonts,amsthm}
\usepackage{mathtools}
\numberwithin{equation}{section}
\allowdisplaybreaks

\usepackage{enumitem}
\usepackage{booktabs}
\usepackage{cite}
\usepackage{tikz}
\usetikzlibrary{arrows.meta,positioning}

\usepackage[
 colorlinks=true,
 linkcolor=blue,
 citecolor=blue,
 urlcolor=blue,
 pdfencoding=auto,
 psdextra,
 pdfusetitle,
 pdfkeywords={complex hyperbolic geometry, equivariant harmonic maps, minimal surfaces, parabolic Higgs bundles, unipotent monodromy, cusp metrics}
]{hyperref}

\newtheorem{theorem}{Theorem}[section]
\newtheorem{prop}[theorem]{Proposition}
\newtheorem{lemma}[theorem]{Lemma}
\newtheorem{cor}[theorem]{Corollary}
\newtheorem*{theorem*}{Theorem}
\theoremstyle{definition}
\newtheorem{definition}[theorem]{Definition}
\newtheorem*{definition*}{Definition}

\newtheorem{remark}[theorem]{Remark}

\newtheorem{notation}[theorem]{Notation}

\newcommand{\C}{\mathbb C}
\newcommand{\Z}{\mathbb Z}
\begin{document}
\raggedbottom
\title[Finite-energy minimal immersions in $\mathbb{CH}^2$]
{Finite-energy equivariant minimal immersions in $\mathbb{CH}^2$ with fixed-lift proper cusp ends}

\author[I. Biswas]{Indranil Biswas}
\address{Department of Mathematics, Shiv Nadar University, Dadri 201314, Uttar Pradesh, India.}
\email{indranil.biswas@snu.edu.in}
\email{indranil29@gmail.com}

\author[P. Kumar]{Pradip Kumar}
\address{Department of Mathematics, Shiv Nadar University, Dadri 201314, Uttar Pradesh, India.}
\email{pradip.kumar@snu.edu.in}

\author[J. Loftin]{John Loftin}
\address{Department of Mathematics and Computer Science, Rutgers-Newark, Newark, NJ 07102, USA}
\email{loftin@rutgers.edu}
\date{}



\hypersetup{
  pdftitle={Finite-energy equivariant minimal immersions in complex hyperbolic two-space with fixed-lift proper cusp ends},
  pdfauthor={Indranil Biswas, Pradip Kumar, and John Loftin},
  pdfsubject={Complex hyperbolic minimal surfaces and parabolic Higgs bundles},
  pdfkeywords={complex hyperbolic geometry, equivariant harmonic maps, minimal surfaces, parabolic Higgs bundles, unipotent monodromy, cusp metrics}
}

\begin{abstract}
Let \(\Sigma_0\,=\,\Sigma\setminus\{z_1,\,\ldots,\, z_n\}\) be a finite-type hyperbolic Riemann
surface, with its universal cover being identified with \(\mathbb H\). We prove that a
finite-energy equivariant harmonic map
\[
f\colon\mathbb H\;\longrightarrow\;\mathbb{CH}^2
\]
is fixed-lift proper at every puncture --- that is, its restriction to a fixed
lifted cusp eventually leaves every compact subset of the target --- if and
only if its peripheral holonomy is parabolic. For conformal immersions,
parabolic holonomy at every puncture also implies completeness of the
descended induced metric.

We further show that if a polystable parabolic
\(\mathrm{PU}(2,\,1)\)-Higgs bundle has local data
\(\alpha_p\,=\,s_p\,=\,0\) and \(Y_p\,\ne\,0\), and if the associated harmonic map is
weakly conformal, then the corresponding end is unbranched, complete, and
has finite energy.

We also show that, for mixed conformal data, stability is equivalent to two
explicit parabolic slope inequalities. Finally, for every \(n\;\ge\;5\), we
construct explicit Higgs data producing mixed unbranched
\(\mathbb{CH}^2\)-\(n\)-noids with nonidentity unipotent peripheral
monodromy, finite total energy, complete induced metric, and fixed-lift
proper ends.
\end{abstract}
\maketitle
\tableofcontents

\section{Introduction}\label{sec:introduction}

Minimal surfaces in the complex hyperbolic plane connect K\"ahler
geometry, harmonic-map theory, and the geometry of surface-group
representations.  On a closed Riemann surface, nonabelian Hodge theory
relates reductive representations to polystable Higgs bundles and
equivariant harmonic maps
\cite{Corlette1988, Donaldson1987, hitchin_1987, Simpson1992}.  A harmonic map
from a Riemann surface is weakly conformal exactly when its Hopf
differential vanishes; such a nonconstant map is a branched minimal map
and is a minimal immersion away from its branch locus.  For
\(\mathbb{CH}^2\), the Higgs-bundle description of equivariant minimal
surfaces, including the mixed case (both complex-type components are
nonzero) and the minimal Lagrangian case, was developed
by Loftin and McIntosh
\cite{loftin_mcintosh_2013,loftin_mcintosh_2019,
loftin_mcintosh_correction_2025}.

Punctures raise several issues that do not occur in the compact setting. Although finite 
energy restricts the possible peripheral holonomy, it does not by itself determine 
whether an equivariant surface escapes along a fixed lifted cusp, or whether the induced 
metric is complete, or whether it is comparable with the standard cusp metric. Likewise, 
an abstract existence theorem for an adapted harmonic metric does not automatically give 
explicit complete examples. The main objectives of this article are therefore geometric:

\begin{enumerate}[label=\textup{(\arabic*)},leftmargin=2.8em]
\item Determine exactly which peripheral holonomy type detects fixed-lift
properness at a cusp and when this properness implies completeness;

\item convert nilpotent parabolic residue data into a quantitative metric
asymptotic at the puncture;
\item construct explicit Higgs data that produce families of mixed minimal
surfaces with these properties.
\end{enumerate}

Let \(\Sigma\) be a compact Riemann surface. Fix a nonempty finite subset
\(P\,\subset\,\Sigma\) with
\(2g(\Sigma)-2+|P|\,>\,0\), and put \(\Sigma_0\,=\,\Sigma\setminus P\).  We
consider a reductive representation 
\[
\rho\ :\ \pi_1(\Sigma_0)\ \longrightarrow\ \mathrm{PU}(2,\,1)
\]
(meaning the closure of $\rho(\pi_1(\Sigma_0))$ is a reductive subgroup)
and a finite-energy \(\rho\)-equivariant harmonic map
\[
f\ :\ \mathbb H\ \longrightarrow\ \mathbb{CH}^2,
\]
after identifying the universal cover with \(\mathbb H\). Global properness is generally not the right notion for an equivariant map. Instead, near a puncture we choose one connected component of the lifted punctured neighbourhood and ask whether the map escapes to infinity along that fixed lift. More precisely, every sequence tending to the puncture within that component must leave every compact subset of \(\mathbb{CH}^2\). This is the fixed-lift properness condition of Definition~\ref{def:proper-puncture}.

\medskip 

Our first main result concerns the relation between peripheral holonomy and fixed-lift properness.  Sagman's finite-energy
criterion excludes axial peripheral holonomy \cite{Sagman2023}.   In
Theorem~\ref{thm:peripheral-classification}, we prove
\[
f\text{ is fixed-lift proper at }p
\quad\Longleftrightarrow\quad
\rho(c_p)\text{ is parabolic}.
\]
For conformal harmonic immersions, parabolic holonomy at every puncture
also implies completeness of the descended induced metric
(Corollary~\ref{cor:proper-and-complete-parabolic}).  

\medskip

The second main result is the quantitative cusp theorem in
Section~\ref{sec:parabolic-ends}. The tame nonabelian Hodge theory of
Simpson, Biquard, Mochizuki, and
Biquard--Garc\'{\i}a-Prada--Mundet i Riera provides the adapted harmonic
reduction, its local asymptotic model, and also the relation between parabolic
data and flat monodromy
\cite{Biquard1997,BGM2020,Mochizuki2006,Simpson1990}. Using this local
model together with the \(\mathrm{PU}(2,\,1)\) energy--Higgs identity and weak
conformality, we obtain a two-sided estimate for the induced surface
metric. More precisely, when
\[
\alpha_p\,=\,0,\qquad s_p\,=\,0,\qquad Y_p\,\ne\,0,
\]
the peripheral monodromy is nonidentity unipotent and
\[
f^*g_{\mathbb{CH}}\ \asymp\
\frac{|dz|^2}{|z|^2\log^2(1/|z|)}.
\]
Here \(A\,\asymp\,B\) means that \(cB\;\le\;A\;\le\;CB\) for some constants
\(c,C\,>\,0\); for metrics, the inequalities are understood as inequalities
of quadratic forms. Thus the induced metric is uniformly comparable to the
standard complete hyperbolic cusp metric. In particular, the map is
unbranched near the puncture, has finite local energy, and every path tending
to the puncture has infinite length with respect to the induced metric.

Theorem~\ref{thm:zero-weight-nilpotent-cusp} collects these conclusions.
When the same hypotheses hold at every puncture, the local cusp
estimate combines with the peripheral-holonomy criterion to give finite total
energy and fixed-lift properness at every end. Thus the local residue
conditions, together with weak conformality, determine the induced metric
up to uniform equivalence with the standard hyperbolic cusp metric.

\medskip

The third result gives explicit Higgs data for which the preceding
conditions occur simultaneously.  For every set of \(n\;\ge\;5\) distinct
points \(P\subset\mathbb{CP}^1\), Section~\ref{sec:n-noid} constructs
stable zero-weight logarithmic Higgs data whose residues have Type~II
nilpotent Jordan form, meaning one nilpotent Jordan block of size three.
Existence theory then produces an associated mixed, unbranched,
equivariant conformal minimal immersion with nontrivial unipotent
peripheral monodromy, finite total energy, fixed-lift proper ends, and a
complete descended metric (Theorem~\ref{thm:CH2-n-noid}).  For these
explicit data, \(n\;\ge\;5\) is exactly the range where both stability
inequalities hold. 

\medskip
A further structural result is obtained in
Section~\ref{sec:alg-characterization}.  For mixed conformal data, the Higgs
field determines a canonical nilpotent flag, and this sharply restricts the
possible invariant subbundles that enter the stability condition.  As a
consequence, the full parabolic stability problem reduces to two explicit
slope inequalities. This gives a concrete algebraic criterion for deciding
when mixed conformal Higgs data produce an adapted harmonic reduction and
consequently an equivariant harmonic map to \(\mathbb{CH}^2\).

\bigskip

Sections~\ref{sec:Preliminaries}--\ref{sec:peripheral-properness} establish the geometric criterion for 
fixed-lift properness. Sections~\ref{sec:parabolic-setup}--\ref{sec:alg-characterization} develop the 
correspondence between the geometric data and the associated parabolic Higgs-bundle data. 
Section~\ref{sec:parabolic-ends} uses the local triple \((\alpha_p,s_p,Y_p)\) to determine the peripheral 
monodromy and the asymptotic cusp geometry, while Section~\ref{sec:n-noid} applies these results to explicit 
Higgs data and constructs the corresponding geometric examples. Appendix~\ref{app:parabolic-details} gives the 
normalisation and parahoric calculations needed in our \(\mathrm{PU}(2,\,1)\) setting, obtained by
specialising the general framework of \cite{BGM2020}.

\smallskip 

The parabolic nonabelian Hodge-theoretic framework used throughout this article draws substantially on \cite{BGM2020}.

\section{Complex hyperbolic geometry}\label{sec:Preliminaries}

\subsection{The projective model and Bergman metric}

Let $\C^{2,1}$ denote $\C^{3}$ equipped with the Hermitian form of signature $(2,\,1)$
\begin{equation}\label{eq:Hermitian21}
\big\langle (Z_1,\,Z_2,\,Z_3)^t,\,(W_1,\,W_2,\,W_3)^t\big\rangle_{2,1}
\,=\,
Z_1\overline{W}_1+Z_2\overline{W}_2-Z_3\overline{W}_3.
\end{equation}
Denote by $\mathbb P(\C^{2,1})$ the complex projective space that parametrises the lines in $\C^{2,1}$;
for any nonzero $Z\,\in\,\C^3$, the corresponding point of
$\mathbb P(\C^{2,1})$ is denoted by $[Z]$. The two-dimensional complex hyperbolic space is
\begin{equation}\label{eq:CH2proj}
\mathbb{CH}^{2}\,:=\,
\bigl\{[Z]\,\in\,\mathbb P(\C^{2,1})\;\big|\;\langle Z,\,Z\rangle_{2,1}\,<\,0\bigr\}.
\end{equation}
It is equipped with a \(\mathrm{PU}(2,\,1)\)-invariant K\"ahler
metric, denoted by \(g_{\mathbb{CH}}\), and is thereby identified with the
rank-one Hermitian symmetric space
\[
\mathbb{CH}^2\,\cong\,\mathrm{SU}(2,\,1)/\mathrm{S}(\mathrm{U}(2)\times \mathrm{U}(1)).
\]
The complex structure of \(\mathbb{CH}^2\) is denoted by \(J\).  For all
\(x\,\in\,\mathbb{CH}^{2}\) and \(X,\,Y\,\in\,T_x\mathbb{CH}^{2}\), we set
\[
\omega(x)(X,Y)\,=\,g_{\mathbb{CH}}(x)(JX,Y).
\]
The associated Hermitian form on \(\mathbb{CH}^{2}\), which we denote by
\(h_B\), is
\[
h_B\,=\,g_{\mathbb{CH}}-\sqrt{-1}\,\omega,
\]
so \(g_{\mathbb{CH}}\,=\,\operatorname{Re}(h_B)\). We normalise
\(g_{\mathbb{CH}}\) so that it has constant holomorphic sectional curvature
\(-4\). The K\"ahler manifold \((\mathbb{CH}^2,\, g_{\mathbb{CH}})\) is
complete and K\"ahler-Einstein. The subscript $B$ in \(h_B\) distinguishes it
from the harmonic reduction \(h\) used later.

\subsection{Isometries of $\mathbb{CH}^2$}

Define
\[
\mathrm{U}(2,\,1)\,=\,\Bigl\{
A\,\in\,\mathrm{GL}_3(\C)\;\Big|\;\langle A(Z),\,A(W)\rangle_{2,1}
\,=\,\langle Z,\,W\rangle_{2,1}\quad\forall\,Z,\,W\,\in\,\C^3
\Bigr\},
\]
and set $\mathrm{SU}(2,\,1)\,=\,\mathrm{U}(2,\,1)\cap \mathrm{SL}_3(\C)$. The projective unitary group is
\[
\mathrm{PU}(2,\,1)\,=\,\mathrm{U}(2,\,1)/(\text{centre})
\,=\,\mathrm{SU}(2,\,1)/(\Z/3\Z).
\]
The standard linear action of $\mathrm{U}(2,\,1)$ on $\C^{2,1}$ descends to an action of
$\mathrm{PU}(2,\,1)$ on $\mathbb P(\C^{2,1})$; this action of $\mathrm{PU}(2,\,1)$ on ${\mathbb P}
(\C^{2,1})$ preserves the domain $\mathbb{CH}^2$ in \eqref{eq:CH2proj}. The resulting action of
$\mathrm{PU}(2,\,1)$ on $\mathbb{CH}^2$ is transitive and consists of holomorphic
isometries of $(\mathbb{CH}^2,\,
g_{\mathbb{CH}})$. Conversely, every holomorphic isometry of $\mathbb{CH}^2$ is given by
the action of some element of $\mathrm{PU}(2,\,1)$.
The full isometry group (allowing anti-holomorphic isometries) is generated by $\mathrm{PU}(2,\,1)$
together with the complex conjugation which is induced by the automorphism $A\, \longmapsto\,
\overline{A}$ of $\mathrm{SU}(2,\,1)$.

\subsubsection{Classification of elements of $\mathrm{PU}(2,\,1)$}

Let $A\,\in\,\mathrm{PU}(2,\,1)$.
The fixed-point locus of $A$ on $\mathbb{CH}^{2}\,\cup\,\partial\mathbb{CH}^{2}$ gives the following trichotomy, analogous to real hyperbolic geometry:
\begin{itemize}
\item \textbf{Elliptic:} $A$ fixes at least one point of $\mathbb{CH}^2$.

\item \textbf{Parabolic:} $A$ fixes exactly one point of the ideal boundary
$\partial\mathbb{CH}^2$ and fixes no point in $\mathbb{CH}^2$.

\item \textbf{Loxodromic} (often called {hyperbolic} in the literature):
$A$ fixes exactly two points of the boundary $\partial\mathbb{CH}^2$.
These points determine a unique real geodesic joining them, which is the
axis of \(A\), and a unique complex geodesic that contains this real axis.  The
restriction of the automorphism $A$ to the complex geodesic is a translation along the real axis; any
rotational component acts in the normal directions around that axis.
\end{itemize}
For comparison with the terminology used below, define the translation
length by
\[
\ell(A)\,:=\,\inf_{q\in\mathbb{CH}^2}
d_{g_{\mathbb{CH}}}(q,Aq).
\]
For \(\mathrm{PU}(2,\,1)\), one has \(\ell(A)\,>\,0\) exactly for loxodromic
elements; on the other hand, \(\ell(A)\,=\,0\) with the infimum attained exactly for elliptic
elements; and \(\ell(A)\,=\,0\) with the infimum unattained exactly for
parabolic elements.  Thus an axial isometry of positive translation length,
called hyperbolic by Sagman, is precisely a loxodromic element here.
See \cite[Section~3.5]{ParkerNotes} for the details.

\subsection{Curvature and CAT$(0)$ features}\label{subsec:CAT}

The real sectional curvatures of $(\mathbb{CH}^2,\,g_{\mathbb{CH}})$ satisfy
\[
-4\;\le\;K_{\mathrm{sec}}\;\le\;-1.
\]
In particular, $(\mathbb{CH}^2,\,g_{\mathbb{CH}})$ is complete, simply connected, and has
negative sectional curvature. Hence $\mathbb{CH}^2$ is a Hadamard manifold.
Moreover, $\mathbb{CH}^2$ is a CAT$(-1)$ space with the stated normalisation,
because \(K_{\mathrm{sec}}\;\le\;-1\). In particular,
$(\mathbb{CH}^2,\,g_{\mathbb{CH}})$ is a CAT$(0)$ space.

We use the following facts:
\begin{enumerate}
\item For any fixed point $o\,\in\, \mathbb{CH}^2$, the function
$p\,\longmapsto\, d(o,\,p)$ is convex along geodesics.

\item Let $u$ be a harmonic map to $\mathbb{CH}^2$ from a Riemannian
manifold $Y$. Then the function $y\, \longmapsto\, d(o,\,u(y))$ on $Y$ is weakly
subharmonic.
\end{enumerate}

\section{Finite-energy harmonic maps and fixed-lift properness at punctures}
\label{sec:CompleteCONFORMALMINIMAL}

\begin{notation}\label{rem:notation}
Throughout, \(\Sigma\) is a compact connected Riemann surface of genus \(g\).
Fix \(n\;\ge\;1\) distinct points
\[
P\,:=\,\{z_1,\ldots,z_n\}\subset\Sigma
\]
such that
\begin{equation}\label{eg}
2g-2+n\ >\ 0.
\end{equation}
Set
\[
\Sigma_0\ :=\ \Sigma\setminus P.
\]
The above condition \(2g-2+n\,>\,0\) implies that $\Sigma_0$ is a hyperbolic Riemann surface; in particular,
it admits a complete finite-area hyperbolic metric. We fix such a metric on $\Sigma_0$ and denote it by
$g_\Sigma$; the area form on $\Sigma_0$ for $g_\Sigma$ is denoted by $dA_{g_\Sigma}$.
\end{notation}

Fix a base point \(x_0\,\in\,\Sigma_0\) and the corresponding universal cover of $\Sigma_0$; so
the universal cover has a marked point over $x_0$. Identify the
universal cover of \(\Sigma_0\) with the upper half-plane
\(\mathbb H\,\subset\,\mathbb C\); note that the condition in \eqref{eg} rules out
\(\mathbb C\) and \(\mathbb{CP}^1\). Let
\begin{equation}\label{vp}
\varpi\ \colon\ \mathbb H\;\longrightarrow\;\Sigma_0
\end{equation}
be the universal covering map. Write \(\Gamma\,:=\,\pi_1(\Sigma_0,\,x_0)\). We consider the
left deck transformations for $\varpi$ and identify \(\Gamma\) with a discrete subgroup of
\(\mathrm{PSL}(2,\mathbb R)\). Therefore, \(\Sigma_0\,=\,\Gamma\backslash\mathbb H\).

Consider pairs of the form $(\rho,\,f)$, where
\begin{enumerate}
\item
\begin{equation}\label{er}
\rho\ \colon\ \pi_1(\Sigma_0,\,x_0)\;\longrightarrow\;\mathrm{PU}(2,\,1)
\end{equation}
is a homomorphism (not assumed reductive at this stage; reductivity is
imposed when nonabelian Hodge theory is invoked), and

\item $f$ is a finite-energy \(\rho\)-equivariant harmonic map
\begin{equation}\label{df}
f\ \colon\ \mathbb H\;\longrightarrow\;\mathbb{CH}^2.
\end{equation}
\end{enumerate}
The \(\rho\)-equivariance condition means that
\begin{equation}\label{df2}
f(\gamma\cdot p)\,=\,\rho(\gamma)\cdot f(p)
\end{equation}
for all $\gamma\,\in\, \pi_1(\Sigma_0,\,x_0)$ and $p\,\in\, \mathbb H$.

For a conformal immersion, the pulled-back Riemannian metric
\begin{equation}\label{gf}
g_f\,:=\,f^*g_{\mathbb{CH}}
\end{equation}
on the universal cover $\mathbb H$ of $\Sigma_0$ is $\pi_1(\Sigma_0,\,x_0)$-invariant.
Consequently, $g_f$ descends to a Riemannian metric on the quotient surface
\(\Sigma_0\,=\,\Gamma\backslash\mathbb H\); this descended metric on \(\Sigma_0\) will also
be denoted by $g_f$.

By construction, the Riemannian metric $g_f$ on $\Sigma_0$ lies in the conformal class 
determined by the complex structure on $\Sigma_0$. We say that {$f$ is complete on 
$\Sigma_0$} if the descended metric $g_f$ on $\Sigma_0$ is complete.

\begin{definition}\label{def:proper-puncture}
As in \eqref{vp}, $\varpi:\mathbb H\longrightarrow \Sigma_0$ is the universal cover. Let
$f:\mathbb H\longrightarrow \mathbb{CH}^2$ be a \(\rho\)-equivariant map.
Take $p\in P$ (see Notation~\ref{rem:notation}), and let
$U\, \subset\, \Sigma_0$ be a disc neighbourhood of $p$; set $U^*\,=\, 
U\setminus\{p\}$. Choose one connected component
\[
\widetilde U^*\ \subset\ \varpi^{-1}(U^*).
\]
We say that $f$ is \emph{fixed-lift proper at \(p\)} if for some
(hence any) base-point
$o\,\in\,\mathbb{CH}^2$, and for every sequence
$\{z_k\}^\infty_{k\,=\,1}$, with all $z_k\,\in\,\widetilde U^*$, and satisfying
the condition $\varpi(z_k)\,\longrightarrow\,p$, the following holds:
\begin{equation}\label{f2}
d_{g_{\mathbb{CH}}}\bigl(f(z_k),\,o\bigr)\;\longrightarrow\;+\infty.
\end{equation}
\end{definition}

In the above definition, the component $\widetilde U^*$ is fixed while the sequence 
varies. The definition is independent of this choice: if $\widetilde 
U^{*\prime}\,=\,\gamma\widetilde U^*$ is another component, where
$\gamma\, \in\, \Gamma\, =\, {\rm Gal}(\varpi)$, then $f(\gamma 
z)\,=\,\rho(\gamma)f(z)$, and the fixed isometry $\rho(\gamma)$
clearly takes sequences leaving 
every compact subset of $\mathbb{CH}^2$ to sequences with the same property. It is also 
independent of the choice of $o$ in \eqref{f2} and of the sufficiently small punctured 
neighbourhood $U^*$ of $p$.

\smallskip 

The above definition can equivalently be reformulated as follows. Choose the positive
generator of the stabiliser of \(\widetilde U^*\) and a cusp coordinate in
which
\[
\widetilde U^*\,\simeq\,\mathbb R\times[Y,\, \infty)
\]
for some \(Y\,>\,0\), and the generator acts by
\((x,\,y)\,\longmapsto\,(x+2\pi,\,y)\).  Fixed-lift properness at \(p\)
means that
for every sequence $(x_k,\,y_k)\,\in\, \mathbb R\times[Y,\,\infty)$, $k\, \in\, \mathbb N$,
satisfying the condition $y_k\,\longrightarrow\,\infty$, the following holds:
\[
d_{g_{\mathbb{CH}}}\bigl(f(x_k,y_k),o\bigr)\;\longrightarrow\;\infty.
\]
The above real coordinates $x_k$ are not required to lie in a fixed fundamental interval.

Fixed-lift properness concerns one lift of the end; it is not ordinary
properness of \(f\), and it does not mean
that the equivariant map $\mathbb H\longrightarrow \mathbb{CH}^2$ is globally proper.

\begin{definition}\label{def:complete-puncture}
 We say that \(f\)
is {complete around \(p\)} (or that the end at \(p\) is complete)
if every piecewise $C^1$ curve $\alpha\,:\,[0,\,1)\,\longrightarrow\, U^*$ satisfying the condition
$\alpha(t)\,\longrightarrow\,  p$ as $t\,\longrightarrow \,1$ has infinite \(g_f\)-length.
\end{definition}

Fixed-lift properness at a puncture $p$ is a condition on the {map} $f$, whereas completeness around $p$ is a
condition on the {induced (pullback) metric}
\[
g_f\,=\,f^{*}g_{\mathbb{CH}},
\qquad g_{\mathbb{CH}}\,=\,\operatorname{Re}(h_B).
\]

\begin{lemma}
\label{lem:fixed-lift-proper-complete}
Take $\rho$ as in \eqref{er}.
Let \(f\,:\,\mathbb H\,\longrightarrow\, \mathbb{CH}^2\) be a \(\rho\)-equivariant conformal
immersion.  If \(f\) is fixed-lift proper at every puncture, then the
descended metric \(g_f\,=\,f^*g_{\mathbb{CH}}\) on \(\Sigma_0\) is complete.
\end{lemma}

\begin{proof}
Since \(f\) is an immersion, \(g_f\) is a smooth positive-definite
Riemannian metric; hence the Hopf--Rinow theorem applies.  If it were incomplete, a
finite-length divergent curve would eventually lie in a single punctured
disc \(U^*\) and tend to its puncture.  Lift its tail to one component
\(\widetilde U^*\). Since the remaining tail length tends to zero,
its image under \(f\) is Cauchy for \(g_{\mathbb{CH}}\); completeness of
\(\mathbb{CH}^2\) makes that image converge, and hence remain in a compact
set.  This contradicts the given condition of fixed-lift properness.
\end{proof}

Corollary~\ref{cor:proper-and-complete-parabolic} shows that, for a
finite-energy conformal harmonic immersion with parabolic peripheral
holonomy, both fixed-lift properness and completeness hold.

\subsection{Energy in conformal coordinates}\label{subsec:type-decomp}

Let \(g_{\mathbb{CH}}\,=\,\operatorname{Re}(h_B)\) be the Riemannian metric on
\(\mathbb{CH}^2\).  For any smooth \(\rho\)-equivariant map
\(f\,:\,\mathbb H\,\longrightarrow\, \mathbb{CH}^2\), define
\[
e_{g_\Sigma}(f)\,:=\,\frac12|df|^2_{g_\Sigma,g_{\mathbb{CH}}};
\]
in view of the $\rho$-equivariance condition, the measure
\(e_{g_\Sigma}(f)dA_{g_\Sigma}\) descends to \(\Sigma_0\).  Define
\begin{equation}\label{eq:energy-invariant}
E(f)\,:=\,\int_{\Sigma_0}e_{g_\Sigma}(f)\,dA_{g_\Sigma}.
\end{equation}
More generally, for every measurable \(\Omega\,\subset\,\Sigma_0\), set
\[
E(f;\Omega)\,:=\,\int_\Omega e_{g_\Sigma}(f)\,dA_{g_\Sigma}.
\]
If \(z\,=\,x+\sqrt{-1}y\) is a single-valued conformal coordinate on \(\Omega\),
the quantities \(|f_x|\) and \(|f_y|\) are evaluated using any local lift of
the coordinate neighbourhood; the $\rho$-equivariance condition ensures that their norms
are independent of the choice of lift. Conformal invariance of surface energy gives the following:
\begin{equation}\label{eq:energy-flat-coordinate}
E(f;\Omega)\,=\,\frac12\int_{z(\Omega)}
\bigl(|f_x|_{g_{\mathbb{CH}}}^2+|f_y|_{g_{\mathbb{CH}}}^2\bigr)\,dx\,dy.
\end{equation}
In a lifted logarithmic cusp coordinate for which a generator of the deck transformations is
\(x\,\longmapsto\, x+2\pi\), the corresponding downstairs integral is evaluated on
one fundamental strip, for example
\([0,\, 2\pi)\times[Y,\, \infty)\).  We use the same notation for a specified
domain in \(\mathbb H\), in which case no descent is invoked.
If \(f\) is a conformal immersion and
\[
g_f\,=\,f^*g_{\mathbb{CH}}\,=\,\lambda(z)|dz|^2,
\]
then \(|f_x|_{g_{\mathbb{CH}}}^2\,=\,|f_y|_{g_{\mathbb{CH}}}^2\,=\,\lambda\), and this implies that
\begin{equation}\label{eq:energy-conformal-coordinate}
E(f;\Omega)\,=\,\int_{z(\Omega)}\lambda(z)\,dx\,dy.
\end{equation}

\section{Peripheral holonomy and fixed-lift proper cusp ends}
\label{sec:peripheral-properness}
We use a theorem of Sagman. As in \cite[Section~2.1]{Sagman2023},
``hyperbolic'' means an axial isometry with positive translation length.

\begin{prop}[{\cite[Proposition 3.8]{Sagman2023}}]\label{prop:Sagman}
Let \(M\,=\,\Gamma\backslash\widetilde M\) be a complete, finite-volume
hyperbolic surface, and let
\((X,\,g_X)\) be a \(\mathrm{CAT}(-1)\) Hadamard manifold. For any representation
\(\rho_1\,:\,\Gamma\,\longrightarrow \,\mathrm{Iso}(X,g_X)\), there exists a finite-energy
\(\rho_1\)-equivariant map
\(u\,:\,\widetilde M\,\longrightarrow\, X\) if and only if the following holds: for every {peripheral} element
\(\gamma\,\in\,\Gamma\) (that is, nontrivial elements that can be conjugated into a cusp subgroup),
the isometry \(\rho_1(\gamma)\) is {not} hyperbolic.
\end{prop}

In our set-up, the complement \(\Sigma_0\,=\,\Sigma\setminus P\) carries a fixed
complete finite-area hyperbolic metric, and hence \(\Sigma_0\) is a complete
hyperbolic surface of finite area. Moreover, \((\mathbb{CH}^2,\,g_{\mathbb{CH}})\)
is a Hadamard \(\mathrm{CAT}(-1)\) space and a complete metric space. Thus
Proposition~\ref{prop:Sagman} applies with \(M\,=\,\Sigma_0\) and
\(X\,=\,\mathbb{CH}^2\), yielding the following corollary.

\begin{cor}\label{cor:peripheral-nonhyperbolic}
If there exists a finite-energy \(\rho\)-equivariant map \(f\,:\,\mathbb H\,
\longrightarrow\, \mathbb{CH}^2\), then for every cusp loop \(c\,\in\,\pi_1(\Sigma_0)\),
the conjugacy class of \(\rho(c)\) is non-hyperbolic; in other words, it is
either elliptic or parabolic.
\end{cor}

We show that fixed-lift properness excludes elliptic holonomy.

\subsection{Analytic estimates on a lifted cusp}

Fix a cusp loop \(c\) corresponding to a puncture of \(\Sigma_0\); choose
the component of the lifted cusp stabilised by \(c\), and choose coordinates
in which \(c\) acts by \(x\,\longmapsto\, x+2\pi\).  For \(Y\,>\,0\), set
\begin{equation}\label{eq:def-cusp-strip}
\begin{aligned}
\widetilde S_Y&\,:=\,\mathbb R\times[Y,\,\infty),&
S_Y&\,:=\,\bigl(\mathbb{R}/2\pi\mathbb{Z}\bigr)\times[Y,\,\infty),\\
z&\,=\,x+\sqrt{-1}y,&
ds^{2}&\,=\,\frac{dx^{2}+dy^{2}}{y^{2}}.
\end{aligned}
\end{equation}
A function on \(S_Y\) is represented by a \(2\pi\)-periodic function on
the lifted strip \(\widetilde S_Y\).  The equivariant map \(f\) is defined
on \(\widetilde S_Y\), while the invariant scalar quantities on \(\widetilde S_Y\) descend to
\(S_Y\).

In the rest of this subsection assume that \(\rho(c)\) is elliptic.

Choose a fixed point
\(o\,\in\,\mathbb{CH}^2\). Let
\(f\,:\,\mathbb H\,\longrightarrow\, \mathbb{CH}^2\) be a \(\rho\)-equivariant harmonic map.
Define on \(\widetilde S_Y\) the function
\begin{equation}\label{eq:def-U}
U(x,y)\,:=\,d_{g_{\mathbb{CH}}}\bigl(o,\,f(x,\,y)\bigr).
\end{equation}
If \(f\) is \(\rho\)-equivariant and \(\rho(c)(o)\,=\,o\), then \(U\)
is \(2\pi\)-periodic in \(x\)
\begin{equation}\label{eq:U-periodic}
U(x+2\pi,\, y)\,=\,U(x,\,y),
\end{equation}
and hence $U$ descends to \(S_Y\).

Let \(r_o(q)\,=\,d_{g_{\mathbb{CH}}}(o,q)\).  Since \(\mathbb{CH}^2\) is
Hadamard, the function \(r_o\) is convex and \(1\)-Lipschitz, and it is smooth off
\(o\).  Since \(f\) is harmonic, the function \(U\,=\,r_o\circ f\) is weakly subharmonic.
Subharmonicity is unchanged by conformal rescaling of a surface metric:
If \(\widehat g\,=\,e^{2\psi}g\), then
\(\Delta_{\widehat g}\,=\,e^{-2\psi}\Delta_g\); this is in the weak sense.  Thus
we may use the flat equation \(U_{xx}+U_{yy}\;\ge\;0\) below.  The
\(1\)-Lipschitz property also gives
\begin{equation}\label{eq:distance-derivative-bounds}
|U_x|\;\le\;|f_x|_{g_{\mathbb{CH}}},
\qquad
|U_y|\;\le\;|f_y|_{g_{\mathbb{CH}}}
\end{equation}
wherever \(U\) is differentiable.

\begin{lemma}\label{lem:cusp-control}
Let \(U\,\in\, C^2(S_Y)\) be subharmonic on the quotient cusp cylinder in
\eqref{eq:def-cusp-strip}. Equivalently, its representative on
\(\widetilde S_Y\) is subharmonic, and satisfies the condition \(U(x+2\pi,\,y)\,=\,U(x,\,y)\).
Assume the following:
\begin{equation}\label{eq:energy-U-assump}
\int_{Y}^{\infty}\int_{0}^{2\pi}\bigl(|U_x|^{2}+|U_y|^{2}\bigr)\,dx\,dy\,<\,\infty.
\end{equation}
Denote
\begin{equation}\label{eq:def-m}
m(y)\,:=\,\frac{1}{2\pi}\int_{0}^{2\pi}U(x,y)\,dx,
\qquad
\mathcal{A}(y)\,:=\,\left(\int_{0}^{2\pi}U_x(x,y)^2\,dx\right)^{1/2}.
\end{equation}
Then the following two statements hold:
\begin{enumerate}
\item the function \(m\) is bounded above on \([Y,\, \infty)\).
\item for each \(y\;\ge\;Y\),
\[
U(0,\,y)\;\le\;\sup_{t\;\ge\;Y}m(t)+\sqrt{2\pi}\,\mathcal{A}(y).
\]
\end{enumerate}
\end{lemma}

\begin{proof}
First, \(m\) is convex.  Since \(U\) is \(2\pi\)-periodic in \(x\),
\[
\int_0^{2\pi}U_{xx}(x,y)\,dx\,=\,0.
\]
Using subharmonicity \(\Delta U\,=\,U_{xx}+U_{yy}\;\ge\;0\), we compute
\[
m''(y)\,=\,\
\frac{1}{2\pi}\int_{0}^{2\pi}U_{yy}(x,y)\,dx
\,=\,\frac{1}{2\pi}\int_{0}^{2\pi}(U_{xx}+U_{yy})\,dx-\frac{1}{2\pi}\int_{0}^{2\pi}U_{xx}\,dx
\;\ge\;0;
\]
so \(m\) is convex on \([Y,\,\infty)\).

Next,
\[
m'(y)\,=\,\frac{1}{2\pi}\int_0^{2\pi}U_y(x,y)\,dx,
\]
hence by the Cauchy--Schwarz inequality,
\[
|m'(y)|^2\;\le\;\frac{1}{2\pi}\int_0^{2\pi}|U_y(x,y)|^2\,dx.
\]
Integrating over \(y\,\in\,[Y,\,\infty)\), and using \eqref{eq:energy-U-assump}, it follows
that
\begin{equation}\label{ml2}
m'\ \,\in\ \, L^2([Y,\,\infty)).
\end{equation}
If \(m'(y_0)\,>\,0\) for some \(y_0\;\ge\;Y\), then convexity
implies that \(m'(y)\;\ge\;m'(y_0)\,>\,0\) for all \(y\;\ge\;y_0\), which contradicts
\eqref{ml2}. Hence \(m'(y)\;\le\;0\) for all
\(y\;\ge\;Y\), so \(m\) is non-increasing. Hence $m$ is bounded above on \([Y,\,\infty)\). This proves (1).

To prove (2), fix \(y\;\ge\;Y\) and set \(v(x)\,:=\,U(x,y)-m(y)\). Then \[\int_0^{2\pi}v(x)\,dx\,=\,0\]
and \(v'(x)\,=\,U_x(x,\,y)\). Since \(v\) is continuous with mean zero, there exists \(x_0\,\in\,[0,\,2\pi]\)
such that \(v(x_0)\,=\,0\). For any \(x\),
\[
|v(x)|\,=\,\left|\int_{x_0}^{x}v'(t)\,dt\right|\;\le\;\int_0^{2\pi}|v'(t)|\,dt\;\le\;
\sqrt{2\pi}\left(\int_0^{2\pi}|v'(t)|^2\,dt\right)^{1/2}\,=\,\sqrt{2\pi}\,\mathcal{A}(y).
\]
Thus we have
\[
U(0,\,y)\,=\,m(y)+v(0)\;\le\;\sup_{t\;\ge\;Y}m(t)+\sqrt{2\pi}\,\mathcal{A}(y).
\]
This proves (2).
\end{proof}

\begin{lemma}\label{lem:cusp-energy-decay}
Let $f\,:\,\mathbb R\times[Y,\,\infty)\,\longrightarrow\, \mathbb{CH}^2$ be a harmonic map satisfying the
condition
\begin{equation}\label{ea}
f(x+2\pi,\,y)\,=\,A f(x,\,y)
\end{equation}
for an isometry $A$ of $\mathbb{CH}^2$.  Assume that
\[
\int_Y^\infty\int_0^{2\pi}
\bigl(|f_x|_{g_{\mathbb{CH}}}^2+|f_y|_{g_{\mathbb{CH}}}^2\bigr)\,dx\,dy\,<\,\infty.
\]
Then
\begin{equation}\label{eq:uniform-energy-decay}
\sup_{x\in\mathbb R}\bigl(|f_x(x,y)|_{g_{\mathbb{CH}}}^2+|f_y(x,y)|_{g_{\mathbb{CH}}}^2\bigr)
\;\longrightarrow\;0
\qquad\text{as }y\;\longrightarrow\;\infty.
\end{equation}
Consequently,
\begin{equation}\label{eq:uniform-displacement-decay}
\sup_{x\in\mathbb R}d_{g_{\mathbb{CH}}}\bigl(f(x,y),A f(x,y)\bigr)
\;\longrightarrow\;0
\qquad\text{as }y\;\longrightarrow\;\infty.
\end{equation}
\end{lemma}

\begin{proof}
Set \(e_f\,=\,\frac12|df|^2_{|dz|^2,g_{\mathbb{CH}}}\). The harmonicity condition is
invariant under conformal transformations in real dimension two. Therefore, \(f\) is harmonic for the flat
cylinder metric.  We set \(\Delta_0\,:=\,\partial_x^2+\partial_y^2\) and adopt the
curvature convention
\[
R(X,Y)Z\,=\,\nabla_X\nabla_YZ-\nabla_Y\nabla_XZ-\nabla_{[X,Y]}Z,
\qquad
K(X,Y)\,=\,\frac{\langle R(X,Y)Y,X\rangle}{|X\wedge Y|^2}.
\]
With these conventions, the Bochner identity for harmonic maps in a flat
orthonormal frame \(e_1,\,e_2\) is
\[
\Delta_0e_f\,=\,|\nabla df|^2-
\sum_{i,j\,=\,1}^2
\bigl\langle
R^{\mathbb{CH}^2}
\bigl(df(e_i),df(e_j)\bigr)df(e_j),df(e_i)
\bigr\rangle\;\ge\;0,
\]
because the target is of nonpositive sectional curvature. From \eqref{ea},
and the fact that \(A\) is an isometry, it follows that
\(e_f(x+2\pi,\,y)\,=\,e_f(x,\,y)\).

Fix \(0\,<\,r_0\,<\,\pi\).  If \(y\;\ge\;Y+r_0\), then the flat disc
of radius \(r_0\) and centred at \((x,\,y)\) is embedded in the quotient cylinder. The
mean-value inequality gives the following:
\[
e_f(x,y)\;\le\;\frac1{\pi r_0^2}
\int_{B_{r_0}(x,y)}e_f\;\le\;\frac1{\pi r_0^2}
\int_{y-r_0}^{y+r_0}\int_0^{2\pi}e_f(s,t)\,ds\,dt.
\]
The last expression is independent of \(x\) and tends to zero as
\(y\longrightarrow \infty\), because the invariant function \(e_f\) lies in
\(L^1(S_Y)\). Therefore, the periodicity condition
implies that
\[
\sup_{x\in\mathbb R}e_f(x,y)\;\longrightarrow\;0,
\]
which is \eqref{eq:uniform-energy-decay}.

Finally,
\begin{align*}
d_{g_{\mathbb{CH}}}\bigl(f(x,\,y),\,A f(x,\,y)\bigr)
&\,=\,d_{g_{\mathbb{CH}}}\bigl(f(x,\,y),\,f(x+2\pi,\,y)\bigr)\\
&\;\le\;\int_x^{x+2\pi}|f_x(s,\,y)|_{g_{\mathbb{CH}}}\,ds\\
&\;\le\;2\pi\sup_{s\in\mathbb R}|f_x(s,\,y)|_{g_{\mathbb{CH}}}.
\end{align*}
This proves \eqref{eq:uniform-displacement-decay}.
\end{proof}

\subsection{Fixed-lift properness and parabolic peripheral holonomy}

\begin{prop}\label{prop:proper-implies-parabolic}
Let \(f\,:\,\mathbb{H}\,\longrightarrow\,\mathbb{CH}^{2}\) be a finite-energy,
\(\rho\)-equivariant harmonic map, and let \(c\) be a cusp loop.
If \(f\) is fixed-lift proper at the corresponding puncture, then the conjugacy class of
\(\rho(c)\) in $\mathrm{PU}(2,\,1)$ is parabolic.
\end{prop}

\begin{proof}
By Corollary~\ref{cor:peripheral-nonhyperbolic}, the element \(\rho(c)\) is either elliptic or parabolic.
To prove by contradiction, assume that \(A\,:=\,\rho(c)\) is elliptic.

Choose a fixed point \(o\,\in\, \mathbb{CH}^{2}\) for the action of \(A\),
and work on the lifted cusp strip \(\widetilde S_Y\) defined in \eqref{eq:def-cusp-strip}.

Since \(E(f)\,<\,\infty\),
\begin{equation}\label{eq:energy-fx}
\int_{Y}^{\infty}\int_{0}^{2\pi}
\bigl(\lvert f_{x}\rvert_{g_{\mathbb{CH}}}^{2}
+\lvert f_y\rvert_{g_{\mathbb{CH}}}^2\bigr)\,dx\,dy\,<\,\infty.
\end{equation}
Consider \(U\) defined in \eqref{eq:def-U}.  After increasing \(Y\) --- if
necessary --- fixed-lift properness implies that
\[
f(\widetilde S_Y)\cap\overline{B_{g_{\mathbb{CH}}}(o,1)}\,=\,\varnothing.
\]
Indeed, otherwise one could choose a sequence with \(y\,\longrightarrow\, \infty\) whose
image remains in that ball.  Hence \(U\) is smooth on
\(\widetilde S_Y\), and periodicity makes it a smooth subharmonic function
on \(S_Y\). Using the bounds on derivative in \eqref{eq:distance-derivative-bounds} it follows that
\[
|U_x|^2+|U_y|^2\;\le\;|f_x|_{g_{\mathbb{CH}}}^2+|f_y|_{g_{\mathbb{CH}}}^2,
\]
and therefore
\begin{equation}\label{eq:energy-U}
\int_{Y}^{\infty}\int_{0}^{2\pi}\bigl(|U_x|^2+|U_y|^2\bigr)\,dx\,dy\,<\,\infty.
\end{equation}

Consider \(\mathcal{A}(y)\) defined in Lemma~\ref{lem:cusp-control}. By Fubini's theorem
and \eqref{eq:energy-U}, we have \(\mathcal{A}\,\in\, L^2([Y,\,\infty))\).
Every continuous \(L^2\)-function on a half-line has a sequence tending to
zero at infinity (otherwise it is eventually bounded below by a positive
constant).  Hence we may choose \(y_n\,\longrightarrow \,\infty\) such that
\begin{equation}\label{eq:A-to-0}
\mathcal{A}(y_n)\;\longrightarrow\;0.
\end{equation}

Since \(f\) is fixed-lift proper at the puncture, the projected sequence
\((0,\, y_n)\) tends to the puncture, and hence \(f(0,\,y_n)\)
leaves every compact subset of \(\mathbb{CH}^2\). In particular, we have
\begin{equation}\label{eq:U-to-infty}
U(0,\,y_n)\,=\,d_{g_{\mathbb{CH}}}\bigl(o,\,f(0,\,y_n)\bigr)\;\longrightarrow\;+\infty.
\end{equation}

On the other hand, Lemma~\ref{lem:cusp-control} implies that
\begin{equation}\label{a1}
U(0,y_n)\;\le\;\sup_{t\;\ge\;Y}\, m(t)+\sqrt{2\pi}\,\mathcal{A}(y_n).
\end{equation}
The function \(m\) is bounded above and \(\mathcal{A}(y_n)\,\longrightarrow \, 0\), so the right-hand side
of \eqref{a1} is bounded
uniformly in \(n\).
This contradicts \eqref{eq:U-to-infty}. Therefore \(\rho(c)\) is not elliptic.
Consequently \(\rho(c)\) is parabolic.
\end{proof}

We prove the converse: parabolic peripheral holonomy implies fixed-lift
properness.

\begin{prop}
\label{prop:parabolic-implies-proper}
Let \(f\,:\,\mathbb{H}\,\longrightarrow\,\mathbb{CH}^{2}\) be a finite-energy,
\(\rho\)-equivariant harmonic map. If \(\rho(c)\) is parabolic for a cusp loop \(c\),
then \(f\) is fixed-lift proper at the corresponding puncture.
\end{prop}

\begin{proof}
By component-independence in Definition~\ref{def:proper-puncture}, choose
the component \(\widetilde U^*\) stabilised by \(c\), and identify it with
the cusp half-plane \(\widetilde S_Y\).  In these coordinates \(c\) acts by
$(x,\,y)\,\longmapsto\,(x+2\pi,\,y)$. Put
\[
A\,:=\,\rho(c).
\]
The energy formula \eqref{eq:energy-flat-coordinate} shows that the hypotheses
of Lemma~\ref{lem:cusp-energy-decay} are satisfied. Consequently,
\begin{equation}\label{eq:parabolic-displacement-to-zero}
\sup_{x\in\mathbb R}d_{g_{\mathbb{CH}}}\bigl(f(x,\, y),\,A f(x,\,y)\bigr)
\;\longrightarrow\;0
\qquad\text{as }\,y\,\longrightarrow\,\infty.
\end{equation}

Suppose that \(f\) is not fixed-lift proper at this cusp in the sense of
Definition~\ref{def:proper-puncture}.  Then there are \((x_k,\,y_k)\), $k\, \in\, \mathbb N$, in the
fixed component \(\widetilde U^*\), with \(y_k\,\longrightarrow\, \infty\), for which the
distances \(d_{g_{\mathbb{CH}}}(f(x_k,\,y_k),\,o)\) fail to tend to infinity. Passing to a
subsequence on which these distances are bounded, and then using
Hopf--Rinow theorem, we obtain a further subsequence such that
\[
f(x_k,\,y_k)\;\longrightarrow\;q_\infty\,\in\,\mathbb{CH}^2.
\]
On the other hand, \eqref{eq:parabolic-displacement-to-zero} and the continuity
of the displacement function $q\,\longmapsto\, d_{g_{\mathbb{CH}}}(q,\, Aq)$ together imply that
\[
d_{g_{\mathbb{CH}}}(q_\infty,\, Aq_\infty)\,=\,0.
\]
Thus $Aq_\infty\,=\,q_\infty$, which contradicts the fact that a parabolic isometry
has no fixed point in $\mathbb{CH}^2$.  Therefore, $f$ is fixed-lift proper
at the puncture.
\end{proof}

Proposition~\ref{prop:proper-implies-parabolic}
and Proposition~\ref{prop:parabolic-implies-proper} give the following criterion.

\begin{theorem}\label{thm:peripheral-classification}
Let \(f\,:\,\mathbb{H}\,\longrightarrow\,\mathbb{CH}^{2}\) be a finite-energy,
\(\rho\)-equivariant harmonic map.  For each puncture \(p\,\in\, P\), let \(c_p\) be a
peripheral loop about \(p\).  Then \(f\) is fixed-lift proper at \(p\), in
the sense of Definition~\ref{def:proper-puncture}, if and only if
\(\rho(c_p)\) is parabolic. Consequently, \(f\) is fixed-lift proper at
every cusp if and only if every peripheral holonomy element is parabolic.
\end{theorem}

Under parabolic peripheral holonomy, fixed-lift properness and completeness
follow:

\begin{cor}\label{cor:proper-and-complete-parabolic}
Let \(f\,:\,\mathbb H\,\longrightarrow\,\mathbb{CH}^{2}\) be a finite-energy,
\(\rho\)-equivariant conformal harmonic immersion. Assume that for every
peripheral loop \(c\,\in\,\pi_1(\Sigma_0)\), the
holonomy \(\rho(c)\) is parabolic. Then \(f\) is fixed-lift proper at every
cusp, and the induced metric \(g_f\) is complete on \(\Sigma_0\).
\end{cor}

\begin{proof}
Theorem~\ref{thm:peripheral-classification} gives fixed-lift properness at
every puncture.  Lemma~\ref{lem:fixed-lift-proper-complete} then produces
completeness of \(g_f\).
\end{proof}

\section{The parabolic \texorpdfstring{$\mathrm{PU}(2,\,1)$}{PU(2,\,1)} Higgs bundle}
\label{sec:parabolic-setup}

This section studies the parabolic $\mathrm{PU}(2,\,1)$ Higgs bundles. The main reference
used is \cite[Sections~2--7]{BGM2020}. To keep the present discussion focused, we
defer the detailed calculations to
Appendix~\ref{app:parabolic-details}.

\smallskip 

For later use, set
\[
K_{\log}\,:=\,K_\Sigma(P),\qquad
G\,:=\,\mathrm{PU}(2,\,1),\qquad
H\,:=\,\mathrm P\bigl(\mathrm U(2)\times\mathrm U(1)\bigr).
\]
The Lie algebra of $G$ (respectively, $H$) will be denoted by $\mathfrak g$ (respectively, $\mathfrak h$).
With these conventions, \(H\,\simeq\,\mathrm U(2)\), \(H^{\mathbb C}\,\simeq\,
\mathrm{GL}_2(\mathbb C)\). Relative to
\(\mathbb C^3\,=\,\mathbb C^2\oplus\mathbb C\), the complexified isotropy
representation is
\begin{equation}\label{eq:PU21-isotropy-model}
\mathfrak m^{\mathbb C}\,=\,\
\left\{
\begin{pmatrix}0&x\\ \xi&0\end{pmatrix}
\;\big\vert\;x\,\in\,\mathbb C^2,\ \xi\,\in\,(\mathbb C^2)^\vee
\right\}\,\simeq\,\mathbb C^2\oplus(\mathbb C^2)^\vee.
\end{equation}
Under this identification, \(a\,\in\, H^{\mathbb C}\) acts by
\(a\cdot(x,\, \xi)\,=\,(ax,\,\xi a^{-1})\). Let
\(\mathfrak g\,=\,\mathfrak h\oplus\mathfrak m\) be the corresponding
Cartan decomposition.

\subsection{Parabolic weights, reductions, and graded residues}
\label{subsec:weights-residues}
We begin by recording the weight, reduction, and residue data at each marked
point.  Fix a maximal torus \(T\,\subset \,H\), with Lie algebra \(\mathfrak t\),
and a Weyl alcove \(\mathcal A\) such that
\(0\,\in\,\overline{\mathcal A}\).  For explicit calculations and details,
see Appendix~\ref{app:local-data-calculation}.

\begin{definition}\label{def:parabolic-principal-bundle}
A {\it parabolic principal} \(H^{\mathbb C}\)-bundle on \((\Sigma,\,P)\) consists
of a holomorphic principal \(H^{\mathbb C}\)-bundle \(\mathcal E\,\longrightarrow\, \Sigma\)
together with a principal weight
\(\alpha_p\,\in\,\sqrt{-1}\,\overline{\mathcal A}\) and a reduction
\(Q_p\,\in\,\mathcal E_p/P_{\alpha_p}\) at every \(p\,\in\, P\),  where
\[
P_{\alpha_p}\,:=\,
\left\{a\,\in\,H^{\mathbb C}:\;
\exp(t\alpha_p)a\exp(-t\alpha_p)
\text{ remains bounded as }\,t\,\rightarrow\, +\infty\right\}.
\]
Such a parabolic principal \(H^{\mathbb C}\)-bundle will be denoted by \(\mathcal E_*\).
\end{definition}

To make these local data explicit, recall that the complexified Lie algebras
are
\(
\mathfrak g^{\mathbb C}\,\simeq\,\mathfrak{sl}_3(\mathbb C)
\)
and
\(
\mathfrak h^{\mathbb C}\,\simeq\,\mathfrak{gl}_2(\mathbb C)
\).
For the diagonal torus of \(H\,\simeq\,\mathrm U(2)\), multiplication by
\(\sqrt{-1}\) identifies \(\sqrt{-1}\mathfrak t\) with the real diagonal
$2\times 2$ matrices.  Thus, for $\alpha_p\, \in\,\sqrt{-1}\mathfrak t$,
\[
\alpha_p\,=\,\operatorname{diag}(u_p,\,v_p),
\qquad u_p,\,v_p\ \in\ \mathbb R.
\]
All inequalities, absolute values, and integral parts applied to
\(u_p,v_p\) below use this real-coordinate identification.
We use the same symbol for the \(2\times2\) principal weight.  Formulas in
the ambient adjoint representation use the trace-free lift
\[
\jmath(A)\,:=\,\operatorname{diag}(A,\,0)
-\frac{\operatorname{tr}(A)}{3}I_3\ \in\ \mathfrak{sl}_3(\mathbb C),
\]
whereas the split rank-three vector lift uses
\(A\,\longmapsto\,\operatorname{diag}(A,0)\).  Since the two lifts differ by multiplication
with a scalar, they induce the same adjoint action.  The calculation, and its relation to the
conventions in \cite{BGM2020}, are given in
Appendix~\ref{app:local-data-calculation}.

\begin{definition}
\label{def:ordinary-range}
The weight \(\alpha_p\) lies in the ordinary logarithmic range if
\begin{equation}\label{eq:PU21-ordinary-range}
|u_p|\,<\,1,\qquad |v_p|\,<\,1,\qquad |u_p-v_p|\,<\,1.
\end{equation}
Equivalently, every eigenvalue of
\(\operatorname{ad}(\jmath(\alpha_p))\) --- for its action on
\(\mathfrak g^{\mathbb C}\) --- has absolute value less than one.  See
\cite[equation~(4.3)]{BGM2020} and
Appendix~\ref{app:parabolic-details}.
\end{definition}

Henceforth, every principal weight appearing in this section or in
Appendix~\ref{app:parabolic-details} is assumed to lie in the ordinary
logarithmic range.

Under this standing assumption, let \(\mathfrak m^{\mathbb C}_\mu\) denote
the eigenspace of \(\operatorname{ad}(\jmath(\alpha_p))\), acting on
\(\mathfrak m^{\mathbb C}\), with eigenvalue $\mu$, and set
\[
\mathfrak m_{\alpha_p}\,:=\,
\bigoplus_{\mu\;\le\;0}\mathfrak m^{\mathbb C}_\mu,
\qquad
\mathfrak m^0_{\alpha_p}\,:=\,\mathfrak m^{\mathbb C}_0.
\]
In a frame compatible with \(Q_p\), the ordinary-range parabolic isotropy
sheaf \(\mathcal P\mathcal E(\mathfrak m^{\mathbb C})\) consists of the
local holomorphic sections whose value at \(p\) lies in
\(\mathfrak m_{\alpha_p}\).

\begin{definition}
\label{def:parabolic-PU21-Higgs}
A parabolic Higgs field on \(\mathcal E_*\) is a section
\[
\Phi\ \in\ H^0\!\left(
\Sigma,\, \mathcal P\mathcal E(\mathfrak m^{\mathbb C})\otimes K_{\log}
\right).
\]
A parabolic \(\mathrm{PU}(2,\,1)\)-Higgs bundle is a pair
\((\mathcal E_*,\,\Phi)\), where $\Phi$ is a parabolic Higgs field on \(\mathcal E_*\).
\end{definition}

On a sufficiently small open neighbourhood of \(p\), choose a centred coordinate \(z\) and a
frame compatible with the reduction.  In this frame,
\[
\Phi\,=\,\bigl(R_p+O(z)\bigr)\frac{dz}{z},\, \ \
R_p\ \in\ \mathfrak m_{\alpha_p},
\]
and the graded residue is represented by
\begin{equation}\label{eq:GradedResiProjec}
\operatorname{Gr}\operatorname{Res}_p(\Phi)\,:=\,\
\operatorname{pr}_{\mathfrak m^0_{\alpha_p}}(R_p).
\end{equation}
Its commuting Jordan decomposition is
\begin{equation}\label{eq:graded-residue-decomposition}
\operatorname{Gr}\operatorname{Res}_p(\Phi)\,=\,s_p+Y_p,
\qquad [s_p,\,Y_p]\,=\,0,
\end{equation}
where \(s_p\) is semisimple and \(Y_p\) is nilpotent.

\begin{definition}\label{def:principal-local-data}
Let \(L_{\alpha_p}\,:=\,Z_{H^{\mathbb C}}(\alpha_p)\).  The
\emph{principal local data} consist of the reduction \(Q_p\) (see
Definition~\ref{def:parabolic-principal-bundle}), the weight
\(\alpha_p\), and the \(L_{\alpha_p}\)-orbit of the graded residue in
\(\mathfrak m^0_{\alpha_p}\).  Once a representative of this orbit has been
chosen, and with \(Q_p\) understood, we write these data as
\[
(\alpha_p,\ s_p,\ Y_p)
\]
(see \eqref{eq:graded-residue-decomposition}).
Only the orbit, not the displayed representative, is intrinsic.
\end{definition}

The explicit weight-space and residue formulas are given in
Appendix~\ref{app:local-data-calculation}.

\subsection{Hecke normalisation and the normalised rank-three representative}
\label{subsec:fixed-normalized-representative}

To formulate degrees and slopes concretely, we now pass from the principal
\(\mathrm{PU}(2,\,1)\)-Higgs-bundle data to a normalised rank-three
representative.  Using the standard representation of
\(H^{\mathbb C}\,\simeq\,\mathrm{GL}_2(\mathbb C)\), the principal $H^{\mathbb C}$-bundle
${\mathcal E}_0$ produces the associated vector bundle
\[
V_0\,:=\,\mathcal E_0\times_{\mathrm{GL}_2(\mathbb C)}\mathbb C^2,
\]
and the isotropy splitting (see \eqref{eq:PU21-isotropy-model}) gives a decomposition
\[
\mathcal E_0\times_{H^{\mathbb C}}\mathfrak m^{\mathbb C}
\,\simeq\,V_0\oplus V_0^*.
\]
Thus the Higgs field on \({\mathcal E}_0\) has exactly the two components
\[
\Phi_{1}\colon\mathcal O_{\Sigma_0}\;\longrightarrow\;
V_0\otimes K_{\Sigma_0},\qquad
\Phi_{2}\colon V_0\;\longrightarrow\;K_{\Sigma_0}.
\]

Via the standard representation
\(H^{\mathbb C}\,\simeq\,\mathrm{GL}_2(\mathbb C)\to\mathrm{GL}(\mathbb C^2)\),
the underlying holomorphic principal bundle determines a rank-two
holomorphic vector bundle.  We denote its extension across \(p\), before
the Hecke normalisation below, by \(V^{\mathrm{raw}}\); this is the raw
holomorphic \(\mathrm{GL}_2(\mathbb C)\)-lattice.  In a compatible diagonal
frame \(e_p\,=\,(e_{p,1},e_{p,2})\), its weights (using the notation of \cite{BGM2020}) are
\[
w_{p,1}\,:=\,u_p,\qquad w_{p,2}\,:=\,v_p.
\]
Define
\begin{equation}\label{eq:main-normalized-weights}
\begin{aligned}
k_{p,j}&\,:=\,-\lceil w_{p,j}\rceil,&
a_{p,j}&\,:=\,\lceil w_{p,j}\rceil-w_{p,j}\in[0,1),\\
\nu_p^{\mathrm{nor}}
&\,:=\,\operatorname{diag}(k_{p,1},\, k_{p,2}).
\end{aligned}
\end{equation}
The above \(a_{p,j}\) are the normalised vector-bundle weights, counted with
multiplicity and listed in increasing order when a filtration is displayed.
The conversion process is summarised by
\[
(\mathcal E_*,\, \Phi)\,
\longmapsto\, (V^{\mathrm{raw}},\,\Phi^{\mathrm{raw}})\,
\xrightarrow{\ \nu^{\mathrm{nor}}\ }\, (V_*,\,\Phi)\,
\longmapsto\, (E_*\,=\,V_*\oplus\mathcal O_*,\, \Phi).
\]
Thus one first passes from the principal parabolic bundle
\((\mathcal E_*,\Phi)\), with weights \(w_{p,i}\), to the associated raw
vector bundle \((V^{\mathrm{raw}},\Phi^{\mathrm{raw}})\) in its meromorphic
equivalence class.  The Hecke transformation then produces
\((V_*,\Phi)\) with normalised vector-bundle weights, after which adjoining
\(\mathcal O_*\) gives the rank-three representative
\((E_*\,=\,V_*\oplus\mathcal O_*,\Phi)\).  See \cite[Definition~3.2 and
Section~3.3]{BGM2020}; the derivation, sign conversion, and weighted flags are in
Appendix~\ref{app:normalization-details}.

By $\mathcal O_*$ we denote the trivial line bundle equipped with the trivial parabolic structure
(in other words, there are no nonzero parabolic weights).

\begin{definition}\label{def:fixed-normalized-lattice}
The chosen parabolic vector bundle \(V_*\), together with the data
\[
\bigl(\alpha_p,Q_p;\,\nu_p^{\mathrm{nor}},
L_{\alpha_p}\!\cdot(s_p+Y_p)\bigr),
\]
is called the \emph{normalised logarithmic lattice}; the associated parabolic vector bundle
\(E_*\,=\,V_*\oplus\mathcal O_*\) is the \emph{normalised rank-three
representative}. Here \(Q_p\) is already part of the principal bundle \(\mathcal E_*\), and
\(\nu_p^{\mathrm{nor}}\) gives the chosen normalisation.
\end{definition}

Henceforth, all degrees, zero divisors, and parabolic slopes refer
to this single choice.  Thus
\[
E_*\,=\,V_*\oplus\mathcal O_*,
\qquad E\,=\,V\oplus\mathcal O_\Sigma,
\]
and the Higgs field is
\begin{equation}\label{eq:PU21-matrix-Higgs}
\Phi\,=\,
\begin{pmatrix}0&\Phi_1\\ \Phi_2&0\end{pmatrix},
\ \ \,
\Phi_1\,:\,\mathcal O_\Sigma\,\longrightarrow\, V\otimes K_{\log},
\, \ \ \Phi_2\,:\,V\,\longrightarrow\, K_{\log}.
\end{equation}

Here and below, a subbundle means a saturated subsheaf, equipped with the
induced parabolic structure, whose quotient is torsion-free. A split
subbundle has the form
\(F_*\,=\,V'_*\oplus W'_*\), with \(V'_*\,\subseteq\, V_*\) and
\(W'_*\,\in\,\{0,\, \mathcal O_*\}\), and $F_*$ is \(\Phi\)-invariant precisely when
\begin{equation}\label{eq:PU21-split-invariance}
\Phi_1(W')\;\subseteq\;V'\otimes K_{\log},
\qquad
\Phi_2(V')\;\subseteq\;W'\otimes K_{\log}.
\end{equation}

\subsubsection{Stability of the normalised rank-three representative}
\label{subsec:PU21-stability}

We next translate the stability condition of principal bundles into the language of this
representative. For a parabolic vector bundle \(F_*\), write its filtration
at \(p\) as
\[
F_p\,=\,F_{p,1}\,\supsetneq\,\cdots\,\supsetneq\,
F_{p,\ell_p}\,\supsetneq\, F_{p,\ell_p+1}\,=\,0,
\]
and let
\(0\;\le\;a_{p,1}\,<\,\cdots\,<\,a_{p,\ell_p}\,<\,1\) be the parabolic
weights.  The weight $a_{p,j}$ corresponds to the quotient
\(F_{p,j}/F_{p,j+1}\) (in our case $\ell_p\;\le\;2$).  Define
\[
\operatorname{pardeg}(F_*)\,:=\,
\deg F+
\sum_{p\in P}\sum_{j\,=\,1}^{\ell_p}
a_{p,j}\dim(F_{p,j}/F_{p,j+1}),
\qquad \mu_{\mathrm{par}}(F_*)\,:=\,
\frac{\operatorname{pardeg}(F_*)}{\operatorname{rk}F}.
\]

\begin{definition}\label{def:PU21-polystability}
The terms \emph{semistable}, \emph{stable}, and \emph{polystable} will mean
\(0\)-semistable, \(0\)-stable, and \(0\)-polystable in the sense
of \cite[Section~4.2, equations~(4.8)--(4.9)]{BGM2020}. 
\end{definition}

Below are the direct consequences we will need later. 
\begin{prop}\label{prop:PU-vector-stability}
The parabolic \(\mathrm{PU}(2,\,1)\)-Higgs bundle represented by
\((E_*,\,\Phi)\) is stable if and only if
\[
\mu_{\mathrm{par}}(F_*)\,<\,\mu_{\mathrm{par}}(E_*)
\]
for every nonzero proper split parabolic subbundle \(F_*\) satisfying
\eqref{eq:PU21-split-invariance}. Replacing \(<\) by \(\le\) gives semistability.
\end{prop}

\begin{lemma}\label{lem:projective-normalization}
For a parabolic line bundle \(M_*\), and a parabolic vector bundle
\(F_*\) of rank $r$, the standard parabolic tensor product satisfies
\[
\operatorname{pardeg}(F_*\otimes M_*)
\,=\,\operatorname{pardeg}(F_*)+r\operatorname{pardeg}(M_*).
\]
Consequently, a common parabolic line twist preserves every slope
difference, the stability type, and the projectivised parabolic Higgs
bundle.
\end{lemma}

\subsection{The adapted local model}
\label{subsec:adapted-local-model}

We conclude the set-up by describing the local metric model used near each
puncture. Set
\[
J\,:=\,\operatorname{diag}(1,\,1,\,-1),\qquad
\mathfrak g^{\mathbb C}\,=\,\mathfrak{sl}_3(\mathbb C).
\]
The complex-linear Cartan involution of $\mathfrak g^{\mathbb C}$ is
\(\theta(Z)\,=\,JZJ\), and the compact conjugation of $\mathfrak g^{\mathbb C}$ that
commutes with \(\theta\) is
\begin{equation}\label{eq:PU21-compact-conjugation}
\tau_0(Z)\,:=\,-Z^\dagger.
\end{equation}

Fix a puncture \(p\,\in\, P\) with principal local data
\((\alpha_p,s_p,Y_p)\), and define
\[
\mathfrak r_p\,:=\,\
\ker\!\left(
\operatorname{Ad}\bigl(e^{2\pi\sqrt{-1}\jmath(\alpha_p)}\bigr)-1
\right)
\cap\ker(\operatorname{ad}s_p).
\]
Since \(\alpha_p\) lies in the ordinary range (see
Definition~\ref{def:ordinary-range}), every eigenvalue \(\mu\) of
\(\operatorname{ad}(\jmath(\alpha_p))\) satisfies the condition \(|\mu|\,<\,1\).
On the corresponding eigenspace,
\(\operatorname{Ad}(e^{2\pi\sqrt{-1}\jmath(\alpha_p)})\) acts by multiplication by
\(e^{2\pi\sqrt{-1}\mu}\). Therefore, its fixed point locus is exactly the
zero eigenspace, and hence we have
\[
\ker\!\left(
\operatorname{Ad}\bigl(e^{2\pi\sqrt{-1}\jmath(\alpha_p)}\bigr)-1
\right)\,=\,\
\ker\bigl(\operatorname{ad}\jmath(\alpha_p)\bigr).
\]
In particular, every element of \(\mathfrak r_p\) commutes with
\(\jmath(\alpha_p)\). This is the ordinary-range specialisation of the
exponential centralisers appearing in
\cite[equations~(4.4) and~(5.4)]{BGM2020}.

First assume that \(Y_p\,\ne\,0\). After replacing \(Y_p\) by a representative
in its graded-residue orbit, the Kostant--Rallis theorem applied in the
reductive symmetric subalgebra \(\mathfrak r_p\) completes \(Y_p\) to a
normalised triple
\begin{equation}\label{eq:PU21-KR-triple}
(H_p,X_p,Y_p),\qquad
[H_p,X_p]\,=\,2X_p,\quad
[H_p,Y_p]\,=\,-2Y_p,\quad
[X_p,Y_p]\,=\,H_p,
\end{equation}
with
\[
H_p\in\sqrt{-1}\,\mathfrak h\cap\mathfrak r_p,\qquad
X_p,Y_p\in\mathfrak m^{\mathbb C}\cap\mathfrak r_p,\qquad
X_p\,=\,-\tau_0(Y_p);
\]
see
\cite[Section~5.1, equations~(5.4)--(5.6), and
Propositions~A.4--A.5]{BGM2020}.
Since \(H_p\,\in\,\mathfrak r_p\), we have
\[
[\jmath(\alpha_p),\,H_p]\,=\,0.
\]
When \(Y_p\,=\,0\), set \(H_p\,=\,X_p\,=\,0\), so the same commutation relation
holds automatically.

This commutation relation simplifies the initial metric in
\cite[Section~5.1]{BGM2020}. In the case where \(Y_p\,\ne\,0\), the formula in
\cite[equation~(5.8)]{BGM2020} contains a conjugation of \(H_p\) by the
angular part determined by \(\alpha_p\). Since
\([\jmath(\alpha_p),\,H_p]\,=\,0\), this conjugation is trivial.

\begin{definition}
\label{def:BGM-initial-model}
Let \(z\) be a coordinate function around \(p\) with $z(p)\,=\,0$. Choose a frame of
$E_*$ around $p$ compatible with the parabolic reduction $Q_p$. The \emph{initial metric} at \(p\) is
\begin{equation}\label{eq:BGM-initial-metric}
h_{0,p}\,=\,\
|z|^{-2\alpha_p}
\bigl(-\log|z|^2\bigr)^{H_p}.
\end{equation}
It is interpreted as an \(H\)-reduction, or equivalently as the associated
Hermitian metric in the normalised rank-three representative. We use
\emph{\(\alpha_p\)-adapted} in the sense of
\cite[Definition~2.10]{BGM2020}.
\end{definition}

Recall that
formula~\eqref{eq:BGM-initial-metric} is written in a compatible
frame of the principal bundle. If the corresponding vector bundle is expressed in a
Hecke-normalised frame, then the Hermitian matrix of \(h_{0,p}\) is transformed
by the usual change-of-frame rule. Thus the Hecke normalisation changes
only the matrix representing the metric and not the underlying principal
local model.

The quasi-isometry class of \(h_{0,p}\) is independent of the auxiliary
local choices. A change of compatible principal frame preserves its
uniform quasi-isometry class; see
\cite[Section~5.1, after equations~(5.3) and~(5.8)]{BGM2020}.
Likewise, replacing \(z\) by another centred coordinate
\[
z'\,=\,u(z)z,\qquad u(0)\,\ne\,0
\]
changes the factors in \eqref{eq:BGM-initial-metric} only by uniformly
bounded terms after shrinking the punctured disc. Hence the resulting
local metric lies in the same uniform quasi-isometry class.

Choose pairwise disjoint coordinate discs around the punctures. On a smaller punctured disc around each \(p\), prescribe the local model \(h_{0,p}\). Since the complement of these smaller punctured discs in \(\Sigma_0\) is compact, the local reductions may be extended and interpolated smoothly over this complement. In this way we obtain a smooth global \(H\)-reduction \(h_0\) on \(\Sigma_0\), which we call the \emph{global initial metric}. Any two such choices that agree with the prescribed local models near the punctures are uniformly quasi-isometric, since they differ only over a compact subset of \(\Sigma_0\).

For a \(0\)-polystable parabolic Higgs bundle equipped with this global
adapted initial metric, \cite[Theorem~5.1]{BGM2020} provides a
Hermitian--Einstein reduction \(h\) of $E_*$ which is \(\alpha_p\)-adapted at every
puncture and uniformly quasi-isometric to \(h_0\). The uniqueness statement
there compares adapted harmonic reductions belonging to the same global
quasi-isometry class. For \(G\,=\,\mathrm{PU}(2,\,1)\), the character-degree
condition in that theorem is vacuous because the Lie algebra is
semisimple.

We fix an admissible invariant form \(B\) on
\(\mathfrak{pu}(2,\,1)\) which is negative definite on \(\mathfrak h\) and positive
definite on \(\mathfrak m\) (as in
\cite[Appendix~A.2]{BGM2020}). Its restriction to \(\mathfrak m\) induces
a fixed positive multiple of the Bergman metric introduced in
Section~\ref{sec:Preliminaries}.

\section{Logarithmic structure of mixed conformal Higgs bundles}
\label{sec:parabolic-higgs}

We now apply the preceding parabolic framework to mixed conformal harmonic
maps.  Under the isotropy splitting
\(\mathfrak m^{\mathbb C}\,=\,\mathbb C^2\oplus(\mathbb C^2)^\vee\) in
\eqref{eq:PU21-isotropy-model}, the \((1,0)\)-part of the differential of
an equivariant harmonic map has two components distinguished by complex
type.  We call the map \emph{mixed} when neither component vanishes
identically; thus, a mixed map is neither holomorphic nor anti-holomorphic.

For comparison, recall the compact case.  For a closed surface of genus
\(g\;\ge\;2\), Loftin and McIntosh showed that a mixed minimal immersion
determines, up to a twist, a Higgs bundle
\[
E\,=\,V\oplus\mathcal O_\Sigma,\qquad
\Phi\,=\,\begin{pmatrix}0&\Phi_1\\ \Phi_2&0\end{pmatrix},
\]
with \(\Phi_1:\mathcal O_\Sigma\longrightarrow  V\otimes K_\Sigma\) and
\(\Phi_2:V\longrightarrow  K_\Sigma\).  Let \(D_1\) and \(D_2\) denote the effective
divisors associated with the zeros of \(\Phi_1\) and \(\Phi_2\), corresponding
respectively to the anti-complex and complex points of the immersion.
Conformality gives \(\Phi_2\Phi_1\,=\,0\) and the exact sequence
\begin{equation}\label{eq:LM-compact-benchmark}
0\longrightarrow K_\Sigma^{-1}(D_1)\longrightarrow V
\longrightarrow K_\Sigma(-D_2)\longrightarrow0,
\end{equation}
whose extension class lies in \(H^1(\Sigma,\, K_\Sigma^{-2}(D_1+D_2))\).
The corresponding Higgs bundle is stable precisely when
\begin{equation}\label{eq:LM-compact-inequalities}
2\deg D_1+\deg D_2\,<\,6(g-1),\qquad
\deg D_1+2\deg D_2\,<\,6(g-1).
\end{equation}
Conversely, such stable extension data with disjoint supports determine a
mixed equivariant minimal immersion.  The same construction also allows
common zeros; in that case, it gives a branched minimal map whose branch
locus is
\(\operatorname{Supp}(D_1)\cap\operatorname{Supp}(D_2)\)
\cite[Theorem~2.3]{loftin_mcintosh_2019}, \cite{loftin_mcintosh_correction_2025}.

\medskip

\subsection{The geometric-to-Higgs dictionary}

We use the parabolic nonabelian Hodge correspondence primarily in the
geometric-to-Higgs direction.

Fix a reductive representation
\(
\rho:\pi_1(\Sigma_0)\longrightarrow\mathrm{PU}(2,\,1)
\)
and a \(\rho\)-equivariant harmonic map
\(
f:\mathbb H\longrightarrow\mathbb{CH}^2
\)
arising from an adapted harmonic reduction.  Denote by
\(
(\mathcal E_*,\,\Phi,\,h)
\)
the corresponding parabolic Higgs bundle together with its adapted harmonic
reduction.  Thus, throughout this subsection, the two collections of data
\[
(\rho,\,f)
\qquad\text{and}\qquad
(\mathcal E_*,\,\Phi,\,h)
\]
describe the same nonabelian Hodge object.

Our aim is to determine the structure of the corresponding logarithmic Higgs
field under the assumption that \(f\) is mixed and weakly conformal.  The
following lemma records the geometric dictionary used for this purpose.

\begin{lemma}
\label{lem:sec6-geometric-dictionary}
Let
\(
(\rho,\,f)\) and \((\mathcal E_*,\,\Phi,\,h)
\)
be corresponding data as above.  With respect to the universal covering
\(\varpi:\mathbb H\longrightarrow \Sigma_0\) fixed above, one has
\begin{equation}\label{eq:sec6-geometric-dictionary}
\begin{aligned}
f\text{ is weakly conformal}
&\Longleftrightarrow
\operatorname{tr}(\Phi^2)\,=\,0,\\
f\text{ is mixed}
&\Longleftrightarrow
\Phi_1\not\equiv0
\text{ and }
\Phi_2\not\equiv0,\\
df_{\widetilde x}\,=\,0
&\Longleftrightarrow
\Phi_1(x)\,=\,0\,=\,\Phi_2(x),
\qquad x\,=\,\varpi(\widetilde x).
\end{aligned}
\end{equation}
\end{lemma}

\begin{proof}
The reduction \(h\) determines a principal \(H\)-subbundle
\[
\mathcal E_h\subset \mathcal E|_{\Sigma_0}
\]
of the principal \(H^{\mathbb C}\)-bundle
\(\mathcal E|_{\Sigma_0}\).  Let
\(
\mathcal E_h(\mathfrak m)
\,:=\,
\mathcal E_h\times_H\mathfrak m.
\)
Its complexification is naturally identified with the restriction of
the isotropy bundle:
\[
\mathcal E_h(\mathfrak m)\otimes_{\mathbb R}\mathbb C
\,\simeq\,
\mathcal E|_{\Sigma_0}
\times_{H^{\mathbb C}}\mathfrak m^{\mathbb C}.
\]

Since
\(
\mathbb{CH}^2\,=\,G/H,
\)
fix once and for all an \(H\)-equivariant real-linear isomorphism
\[
\iota:\mathfrak m\longrightarrow T_{eH}(G/H)
\]
and use its complexification.  Together with the reduction \(h\), this
choice induces the associated-bundle identification
\[
f^*T\mathbb{CH}^2
\,\simeq\,
\varpi^*\mathcal E_h(\mathfrak m).
\]
We use this identification, and its complexification, without further
notation.  Relative to it and to our convention for obtaining the Higgs
field from the flat connection, there is a fixed
\(c_0\in\mathbb C^\times\) such that the relation is exactly
\begin{equation}\label{eq:df-Phi-exact-normalization}
(df)^{1,0}\,=\,c_0\,\varpi^*\Phi.
\end{equation}
The same choice of \(\iota\) and the same scalar \(c_0\) are used
throughout.

Let \(\widetilde x\in\mathbb H\) and put
\(x\,=\,\varpi(\widetilde x)\).  Since \(df\) is real,
\[
df_{\widetilde x}\,=\,0
\quad\Longleftrightarrow\quad
(df)^{1,0}_{\widetilde x}\,=\,0
\quad\Longleftrightarrow\quad
(\varpi^*\Phi)_{\widetilde x}\,=\,0
\quad\Longleftrightarrow\quad
\Phi(x)\,=\,0.
\]
Using the normalised rank-three representative,
\[
\Phi(x)\,=\,
\begin{pmatrix}
0&\Phi_1(x)\\
\Phi_2(x)&0
\end{pmatrix},
\]
so \(\Phi(x)\,=\,0\) is equivalent to
\(
\Phi_1(x)\,=\,0\,=\,\Phi_2(x).
\)
This proves the third equivalence in
\eqref{eq:sec6-geometric-dictionary}.

The two summands in
\[
\mathfrak m^{\mathbb C}
\,\simeq\,
\mathbb C^2\oplus(\mathbb C^2)^\vee
\]
correspond to the two type components of the differential of \(f\).
Hence \(f\) is mixed precisely when neither component vanishes identically;
equivalently,
\[
f\text{ is mixed}
\quad\Longleftrightarrow\quad
\Phi_1\not\equiv0
\text{ and }
\Phi_2\not\equiv0.
\]

It remains to verify the conformality statement.  Under the realisation
\[
\mathfrak m^{\mathbb C}
\,=\,
\left\{
\begin{pmatrix}
0&u\\
\xi&0
\end{pmatrix}
:\;
u\,\in\,\mathbb C^2,\;
\xi\,\in\,(\mathbb C^2)^\vee
\right\},
\]
the complex-bilinear extension of the invariant metric on
\(\mathfrak m\) is a nonzero scalar multiple of
\[
(X,Y)\,\longmapsto\,\operatorname{tr}(XY).
\]
By \eqref{eq:df-Phi-exact-normalization}, the Hopf
differential of \(f\) is a nonzero scalar multiple of
\(\operatorname{tr}(\Phi^2)\) on \(\Sigma_0\).  This quadratic differential
extends as a holomorphic section of \(K_{\log}^2\) on \(\Sigma\), so its
vanishing on \(\Sigma_0\) is equivalent to global vanishing.  Therefore
\[
f\text{ is weakly conformal}
\quad\Longleftrightarrow\quad
\operatorname{tr}(\Phi^2)\,=\,0.
\]
\end{proof}

\subsection{Logarithmic extension}
\label{subsec:sec6-log-extension}
We now extract the global bundle structure forced by conformality.  For an
off-diagonal field,
\(\operatorname{tr}(\Phi^2)\,=\,0\) forces
\(\Phi_2\Phi_1\,=\,0\).  If neither component vanishes identically, the
generic image of \(\Phi_1\) and the generic kernel of \(\Phi_2\) are the
same line in \(V\).  At a zero of either component, however, the first image
need not be a subbundle, and the second map need not be surjective onto
\(K_{\log}\).  Accordingly, we replace the first image by its saturation and
corestrict the second map to its image.

\begin{definition}
\label{def:bundle-valued-zero-divisor}
Let \(a:L\longrightarrow  F\) be a nonzero morphism from a line bundle to a vector
bundle on \(\Sigma\), and set
\[
W(a)\,:=\,\operatorname{im}(a)^{\mathrm{sat}}\subset F.
\]
Define the effective divisor
\(
D(a)\,:=\,
\sum_{x\in\Sigma}
\operatorname{length}_{\mathcal O_{\Sigma,x}}
\bigl(W(a)_x/\operatorname{im}(a)_x\bigr)\,x.
\)
There is a unique identification
\[
W(a)\,\simeq\,L(D(a))
\]
under which \(a:L\longrightarrow  W(a)\) is the canonical inclusion
\(L\hookrightarrow L(D(a))\).
\end{definition}
\begin{definition}
\label{def:line-valued-zero-divisor}
For a nonzero morphism \(b:F\longrightarrow  M\) from a vector bundle to a line bundle,
define
\(
D(b)\,:=\,
\sum_{x\in\Sigma}
\operatorname{length}_{\mathcal O_{\Sigma,x}}
\bigl(M_x/\operatorname{im}(b)_x\bigr)\,x.
\)
Then
\[
\operatorname{im}(b)\,\simeq\,M(-D(b)).
\]
\end{definition}

We apply these definitions to the normalised rank-three representative, in
which the Higgs field has components
\[\Phi_1:\mathcal O_\Sigma\longrightarrow V\otimes K_{\log},
\qquad
\Phi_2:V\longrightarrow K_{\log}.\]
Suppose that \(\Phi_1\) and \(\Phi_2\) are both nonzero.  Let
\(\widetilde\Phi_1:K_{\log}^{-1}\longrightarrow V
\)
be the morphism obtained from \(\Phi_1\) by tensoring with
\(K_{\log}^{-1}\). Since tensoring by a line bundle is exact, it carries
the saturated image of one map to the saturated image of the other and
preserves the corresponding stalk lengths.  We therefore set
\[
D_1\,:=\,D(\widetilde\Phi_1)\,=\,D(\Phi_1),
\qquad
D_2\,:=\,D(\Phi_2).
\]

For a Higgs field in the block form
\eqref{eq:PU21-matrix-Higgs}, a direct multiplication gives
\begin{equation}\label{eq:trace-Phi-square}
\operatorname{tr}(\Phi^2)
\,=\,
\operatorname{tr}(\Phi_1\Phi_2)+\Phi_2\Phi_1
\,=\,
2\Phi_2\Phi_1.
\end{equation}

\begin{prop}
\label{prop:sec6-log-extension}
Let \((\mathcal E_*,\,\Phi)\) be a parabolic
\(\mathrm{PU}(2,\,1)\)-Higgs bundle on \((\Sigma,P)\), all of whose principal
weights lie in the ordinary logarithmic range.  In the normalised
rank-three representative, write
\[
E_*\,=\,V_*\oplus\mathcal O_*,
\qquad
\Phi\,=\,
\begin{pmatrix}0&\Phi_1\\ \Phi_2&0\end{pmatrix},
\]
and assume that
\(
\Phi_1\not\equiv0\), \(
\Phi_2\not\equiv0\), and 
\(\operatorname{tr}(\Phi^2)\,=\,0.
\)
Set
\(
D_1\,:=\,D(\widetilde\Phi_1)\), 
\(D_2\,:=\,D(\Phi_2)\), and 
\(d_i\,:=\,\deg D_i.
\)
Then the following statements hold:
\begin{enumerate}[label=\textup{(\roman*)},leftmargin=2.4em]
\item
\(
\operatorname{im}(\widetilde\Phi_1)^{\mathrm{sat}}
\,=\,\ker(\Phi_2)
\,\simeq\,K_{\log}^{-1}(D_1).
\)

\item There is an exact sequence
\begin{equation}\label{eq:sec6-log-LM}
0
\longrightarrow
K_{\log}^{-1}(D_1)
\longrightarrow
V
\longrightarrow
K_{\log}(-D_2)
\longrightarrow0.
\end{equation}

\item
\(\deg V\,=\,d_1-d_2\), and 
\(\deg\ker(\Phi_2)\,=\,-\deg K_{\log}+d_1.
\)
\end{enumerate}
\end{prop}

\begin{proof}
In the normalised logarithmic lattice, both components are holomorphic
maps on the compact curve:
\[
\widetilde\Phi_1:K_{\log}^{-1}\longrightarrow V,
\qquad
\Phi_2:V\longrightarrow K_{\log}.
\]
Consequently, their kernels, images, saturations, and stalk lengths are
defined at every \(p\in P\), rather than only on \(\Sigma_0\).

By \eqref{eq:trace-Phi-square}, \(\Phi_2\Phi_1\,=\,0\); after tensoring by
\(K_{\log}^{-1}\), this is
\(\Phi_2\circ\widetilde\Phi_1\,=\,0\).
Hence
\(
\operatorname{im}(\widetilde\Phi_1)
\subseteq\ker(\Phi_2).
\)
Because \(\operatorname{im}(\Phi_2)\) is a torsion-free rank-one sheaf,
\(\ker(\Phi_2)\) is a saturated line subbundle of \(V\).  It therefore
contains
\[
 W_1\,:=\,\operatorname{im}(\widetilde\Phi_1)^{\mathrm{sat}}\subseteq \ker(\Phi_2).
\]
At the generic point of \(\Sigma\), both \(W_1\) and \(\ker(\Phi_2)\) are
one-dimensional, so their inclusion is an equality on the generic fibre.
Hence
\(\ker(\Phi_2)/W_1\) is torsion.  On the other hand, it injects into the
torsion-free sheaf \(V/W_1\), because \(W_1\) is saturated.  Thus the
quotient vanishes and
\[
W_1\,=\,\ker(\Phi_2).
\]
Definition~\ref{def:bundle-valued-zero-divisor} gives
\[
W_1\,\simeq\,K_{\log}^{-1}(D_1),
\qquad
\operatorname{im}(\Phi_2)\,=\,K_{\log}(-D_2).
\]
Corestricting \(\Phi_2\) to its image proves
\eqref{eq:sec6-log-LM}.  Taking degrees then gives (iii).
\end{proof}

Here ``canonical'' is understood relative to the normalised logarithmic
lattice and the given Higgs field.  The proposition also determines the
extension class
\[\xi\in
\operatorname{Ext}^1_{\mathcal O_\Sigma}\!\left(
K_{\log}(-D_2),K_{\log}^{-1}(D_1)\right)
\,\simeq\,
H^1\!\left(
\Sigma,K_{\log}^{-2}(D_1+D_2)\right).\] 

The extension and divisor statements above are sheaf-theoretic and require
neither polystability nor an adapted harmonic metric.  Once the full
parabolic structure is fixed, the canonical extension identifies the
invariant subbundles and their slopes.  This identification provides the
starting point for the stability analysis in the next section.

\section{Stability and reconstruction from mixed conformal data}
\label{sec:alg-characterization}

We first identify the canonical invariant flag and use it to reduce stability
to two slope inequalities.  We then express those inequalities in terms of
the divisor and parabolic-weight data and reconstruct an adapted harmonic
reduction, together with its equivariant branched minimal map.

\subsection{The nilpotent flag and the slope test}
With the notation and assumptions of
Proposition~\ref{prop:sec6-log-extension}, let
\[
W_1\,:=\,
\operatorname{im}(\widetilde\Phi_1)^{\mathrm{sat}}
\,=\,\ker(\Phi_2)
\,\simeq\,K_{\log}^{-1}(D_1)
\subset V,
\qquad
W_2\,:=\,W_1\oplus\mathcal O_\Sigma\subset E.
\]
We write \(W_{1,*}\) and \(W_{2,*}\) for their induced parabolic
structures.
\begin{lemma}
\label{lem:sec7-invariant-subbundles}
As subsheaves of
\(E\),
\begin{equation}\label{eq:sec7-canonical-flag}
0
\subset
W_1\,=\,\ker(\Phi)
\subset
W_2\,=\,\ker(\Phi^2)
\subset
E.
\end{equation}
Here all kernels are sheaf kernels.  Moreover, if \(F\subset E\) is a
nonzero proper saturated \(\Phi\)-invariant coherent subsheaf, then
\(
F\,=\,W_1\) or
\(F\,=\,W_2.
\)
\end{lemma}

\begin{proof}
For a local section
\((v,c)\) of \(V\oplus\mathcal O_\Sigma\),
\(
\Phi(v,c)
\,=\,
\bigl(\Phi_1(c),\,\Phi_2(v)\bigr).
\)
Since
\(\widetilde\Phi_1:K_{\log}^{-1}\longrightarrow V
\)
is a nonzero morphism from a line bundle to a vector bundle, it is
injective as a morphism of sheaves.  Equivalently,
\(\Phi_1:\mathcal O_\Sigma\longrightarrow  V\otimes K_{\log}\) is injective.
Consequently,
\[
\Phi(v,c)\,=\,0
\quad\Longleftrightarrow\quad
c\,=\,0\ \text{ and }\ \Phi_2(v)\,=\,0.
\]
Using \(W_1\,=\,\ker(\Phi_2)\), we obtain
\(
\ker(\Phi)\,=\,W_1.
\)

Next, \(\operatorname{tr}(\Phi^2)\,=\,0\) implies
\(
\Phi_2\Phi_1\,=\,0
\)
by \eqref{eq:trace-Phi-square}.  Thus
\[
\Phi^2(v,c)
\,=\,
\bigl(\Phi_1\Phi_2(v),\,0\bigr).
\]
The morphism
\(
\Phi_1\otimes\operatorname{id}_{K_{\log}}
:
K_{\log}\longrightarrow V\otimes K_{\log}^{2}
\)
is injective.  Hence
\[
\Phi_1\Phi_2(v)\,=\,0
\quad\Longleftrightarrow\quad
\Phi_2(v)\,=\,0
\quad\Longleftrightarrow\quad
v\in W_1.
\]
There is no condition on \(c\).  Therefore
\(
\ker(\Phi^2)
\,=\,
W_1\oplus\mathcal O_\Sigma
\,=\,
W_2,
\)
which proves \eqref{eq:sec7-canonical-flag}.  In particular,
\(W_1\) and \(W_2\) are \(\Phi\)-invariant.

It remains to classify saturated \(\Phi\)-invariant subsheaves.
Let
\(
\mathbb K\,:=\,\mathbb C(\Sigma)
\)
be the function field of \(\Sigma\), and let \(\eta\) denote the
generic point.  After choosing a trivialisation of
\(K_{\log,\eta}\), the Higgs field induces a
\(\mathbb K\)-linear endomorphism
\[
N:E_\eta\longrightarrow E_\eta.
\]
The relation \(\Phi_2\Phi_1\,=\,0\) gives \(N^3\,=\,0\).  Moreover, the
identities already proved imply
\[
\ker N\,=\,(W_1)_\eta,
\qquad
\ker N^2\,=\,(W_2)_\eta.
\]
Since \(W_1\) and \(W_2\) have ranks \(1\) and \(2\), respectively,
\[
\dim_{\mathbb K}\ker N\,=\,1,
\qquad
\dim_{\mathbb K}\ker N^2\,=\,2.
\]
Thus \(N^2\,\ne\,0\), and \(N\) is a nilpotent endomorphism of the
three-dimensional vector space \(E_\eta\) having a single Jordan
block of size \(3\).  Equivalently, with \(t\) acting by \(N\),
\[
E_\eta\,\simeq\,\mathbb K[t]/(t^3)
\]
as a \(\mathbb K[t]\)-module.

Now let \(F\subset E\) be as in the statement.  Passing to the
generic fibre gives a nonzero proper \(N\)-invariant subspace
\[
F_\eta\subset E_\eta,
\]
equivalently a nonzero proper
\(\mathbb K[t]\)-submodule of
\(\mathbb K[t]/(t^3)\).  The only such submodules are
\[
(t^2)
\qquad\text{and}\qquad
(t).
\]
Under the preceding identification these are respectively
\[
\ker N\,=\,(W_1)_\eta,
\qquad
\ker N^2\,=\,(W_2)_\eta.
\]
Hence
\[
F_\eta\,=\,(W_1)_\eta
\qquad\text{or}\qquad
F_\eta\,=\,(W_2)_\eta.
\]

Finally, a saturated subsheaf of a vector bundle on a smooth
integral curve is determined by its generic fibre.  More precisely,
for every \(x\in\Sigma\),
\[
F_x\,=\,E_x\cap F_\eta
\qquad\text{inside }E_\eta.
\]
Indeed, \(E_x/F_x\) is torsion-free because \(F\) is saturated; over
the discrete valuation ring \(\mathcal O_{\Sigma,x}\), this is
equivalent to the displayed equality.  The same applies to the
saturated subbundles \(W_1\) and \(W_2\).  Equality of generic fibres
therefore yields
\[
F\,=\,W_1
\qquad\text{or}\qquad
F\,=\,W_2.
\]
\end{proof}
As in \cite{loftin_mcintosh_2019}, for a split vector bundle satisfying the
conformality condition, polystability is equivalent to stability.  Thus,
under the mixed conformal hypotheses, we may use the two terms
interchangeably and formulate the criterion in terms of stability.

\begin{cor}
\label{cor:sec7-stability-reduction}
Under the hypotheses of Proposition~\ref{prop:sec6-log-extension}, the
parabolic \(\mathrm{PU}(2,\,1)\)-Higgs bundle
\((\mathcal E_*,\,\Phi)\) is stable if and only if
\begin{equation}\label{eq:sec7-stab-ineq}
\mu_{\mathrm{par}}(W_{1,*})\,<\,\mu_{\mathrm{par}}(E_*),
\qquad
\mu_{\mathrm{par}}(W_{2,*})\,<\,\mu_{\mathrm{par}}(E_*).
\end{equation}
If the vector representative is changed by a common parabolic line
twist, the corresponding two inequalities are unchanged.
\end{cor}

\begin{proof}
Proposition~\ref{prop:PU-vector-stability} tests stability against nonzero
proper split \(\Phi\)-invariant parabolic subbundles.  By
Lemma~\ref{lem:sec7-invariant-subbundles}, the only possibilities even
without the split restriction are \(W_1\) and \(W_2\), and both are
split:
\[
W_1\,=\,W_1\oplus0,
\qquad
W_2\,=\,W_1\oplus\mathcal O_\Sigma.
\]
The assertion about changing representatives is
Lemma~\ref{lem:projective-normalization}.
\end{proof}

The two slopes can be written in terms of the divisor degrees and the
normalised vector-bundle weights.

\begin{prop}
\label{prop:sec7-explicit-stability}
Let \((\mathcal E_*,\Phi)\) satisfy the hypotheses of
Proposition~\ref{prop:sec6-log-extension}, with notation as above, and put
\[
\kappa\,:=\,\deg K_{\log}\,=\,2g-2+|P|.
\]
For each \(p\in P\), let
\(
a_{p,1},\,a_{p,2}\in[0,1)
\)
be the two normalised vector-bundle weights of the rank-two parabolic
bundle \(V_*\), counted with multiplicity, and set
\[
\omega_p\,:=\,a_{p,1}+a_{p,2}.
\]
Thus the weights of the normalised rank-three representative
\(
E_*\,=\,V_*\oplus\mathcal O_*
\)
are \(a_{p,1},a_{p,2},0\).

Let \(\lambda_p\in[0,1)\) be the weight induced on the parabolic line
\(W_{1,*}\).  Then
\((\mathcal E_*,\,\Phi)\) is stable if and only if
\begin{align}
2d_1+d_2
&\,<\,
3\kappa+\sum_{p\in P}(\omega_p-3\lambda_p),
\label{eq:sec7-ineq1}\\
d_1+2d_2
&\,<\,
3\kappa+\sum_{p\in P}(2\omega_p-3\lambda_p).
\label{eq:sec7-ineq2}
\end{align}
\end{prop}

\begin{proof}
The exact sequence \eqref{eq:sec6-log-LM} gives
\[
\deg E\,=\,\deg V\,=\,d_1-d_2,
\qquad
\deg W_1\,=\,\deg W_2\,=\,-\kappa+d_1.
\]
Since \(\mathcal O_*\) has weight zero, the weights induced on
\(W_{2,*}\,=\,W_{1,*}\oplus\mathcal O_*\) are \(\lambda_p\) and \(0\).
Consequently,
\[
\begin{aligned}
\operatorname{pardeg}(E_*)
&\,=\,d_1-d_2+\sum_{p\in P}\omega_p,\\
\operatorname{pardeg}(W_{1,*})
&\,=\,-\kappa+d_1+\sum_{p\in P}\lambda_p,\\
\operatorname{pardeg}(W_{2,*})
&\,=\,-\kappa+d_1+\sum_{p\in P}\lambda_p.
\end{aligned}
\]
The rank-one and rank-two inequalities in
\eqref{eq:sec7-stab-ineq} are, respectively,
\[
3\operatorname{pardeg}(W_{1,*})\,<\,\operatorname{pardeg}(E_*),
\qquad
3\operatorname{pardeg}(W_{2,*})\,<\,2\operatorname{pardeg}(E_*).
\]
Substitution gives \eqref{eq:sec7-ineq1} and
\eqref{eq:sec7-ineq2}.
\end{proof}

For \(P\,=\,\varnothing\) and \(g\;\ge\;2\),
\(K_{\log}\,=\,K_\Sigma\) and
\(\kappa\,=\,2g-2\).  There are no parabolic-weight terms, and
\eqref{eq:sec7-ineq1}, \eqref{eq:sec7-ineq2} become 
Loftin--McIntosh inequalities
\eqref{eq:LM-compact-inequalities}.  Thus the preceding calculation recovers
the compact criterion. 

\subsection{Reconstruction from stable parabolic Higgs data}

We now use the correspondence in the direction opposite to that used in
Section~\ref{sec:parabolic-higgs}.  There we started with geometric data
\((\rho,\,f)\), together with the corresponding adapted harmonic reduction,
and described the logarithmic structure of the associated parabolic Higgs
bundle.  Here we start with the parabolic Higgs data and reconstruct the
geometric pair.  Combining Sections~\ref{sec:parabolic-higgs}
and~\ref{sec:alg-characterization} gives the synthesis in both directions.

\begin{theorem}
\label{thm:mixed-conformal-correspondence}
The following statements hold.
\begin{enumerate}[label=\textup{(\roman*)},leftmargin=2.8em]
\item \label{item:geometric-side}
Let \(\rho:\pi_1(\Sigma_0)\longrightarrow \mathrm{PU}(2,\,1)\) be reductive, and let
\[
f:\mathbb H\longrightarrow\mathbb{CH}^2
\]
be a \(\rho\)-equivariant harmonic map arising from an adapted harmonic
reduction.  Suppose that \(f\) is mixed and weakly conformal and that its
parabolic Higgs bundle has principal weights in the ordinary logarithmic
range and is stable.  Then its normalised rank-three representative has
the form
\[
E_*\,=\,V_*\oplus\mathcal O_*,
\qquad
\Phi\,=\,
\begin{pmatrix}0&\Phi_1\\ \Phi_2&0\end{pmatrix},
\]
with
\[
\Phi_1\not\equiv0,\qquad
\Phi_2\not\equiv0,\qquad
\operatorname{tr}(\Phi^2)\,=\,0,
\]
and its divisor and weight data satisfy
\begin{align*}
2d_1+d_2
&\,<\,3\kappa+\sum_{p\in P}(\omega_p-3\lambda_p),\\
d_1+2d_2
&\,<\,3\kappa+\sum_{p\in P}(2\omega_p-3\lambda_p).
\end{align*}

\item \label{item:higgs-side}
Conversely, let \((\mathcal E_*,\Phi)\) be a parabolic
\(\mathrm{PU}(2,\,1)\)-Higgs bundle with principal weights in the ordinary
logarithmic range, including its filtrations and principal local data.
Suppose that its normalised rank-three representative has the preceding block
form, that \(\Phi_1\not\equiv0\), that \(\Phi_2\not\equiv0\), that
\(\operatorname{tr}(\Phi^2)\,=\,0\), and that the two mentioned strict
inequalities hold.  Then it admits an adapted harmonic reduction \(h\).
\end{enumerate}

In either direction, the underlying logarithmic rank-two bundle fits
into the canonical exact sequence
\[
0
\longrightarrow
K_{\log}^{-1}(D_1)
\longrightarrow
V
\longrightarrow
K_{\log}(-D_2)
\longrightarrow
0,
\]
with extension class
\[
\xi\in
H^1\!\left(
\Sigma,K_{\log}^{-2}(D_1+D_2)
\right).
\]
With the two end line bundles fixed by \(D_1\) and \(D_2\), this extension
class determines \(V\) up to isomorphism.
The resulting map is an unbranched mixed conformal minimal immersion if
and only if \(\Phi_1\) and \(\Phi_2\) have no common zero on
\(\Sigma_0\).
\end{theorem}

\begin{proof}
For part~\ref{item:geometric-side},
Lemma~\ref{lem:sec6-geometric-dictionary} gives the
block conditions and conformality equation.  Proposition~\ref{prop:sec6-log-extension}
gives the exact sequence and extension class, while stability and
Proposition~\ref{prop:sec7-explicit-stability} give the two strict
inequalities.  The branch-locus statement is the final assertion of the
dictionary lemma.

For part~\ref{item:higgs-side},
Proposition~\ref{prop:sec7-explicit-stability} makes the two strict
inequalities equivalent to stability.  Stability implies
\(0\)-polystability.  With the global adapted initial metric \(h_0\) fixed
in Definition~\ref{def:BGM-initial-model}, the character-degree hypothesis
of \cite[Theorem~5.1]{BGM2020} is vacuous because
\(\mathfrak{pu}(2,\,1)\) is semisimple.  That theorem and
\cite[Section~6.1]{BGM2020} therefore give \(h\), its flat connection, and
the reductive representation \(\rho_h\); the local monodromy data are
related by \cite[Proposition~6.3]{BGM2020}.  The mixed, conformal, and
branch-locus assertions follow from
Lemma~\ref{lem:sec6-geometric-dictionary}.
\end{proof}

\begin{cor}
\label{cor:sec7-conditional-ends}
In the Higgs-to-geometric direction of
Theorem~\ref{thm:mixed-conformal-correspondence}, assume that
\(\Phi_1\) and \(\Phi_2\) have no common zero on \(\Sigma_0\) and that
the corresponding map \(f_h\) has finite energy.  Then \(f_h\) is
fixed-lift proper at \(p\in P\) if and only if
\(\rho_h(c_p)\) is parabolic.

If every peripheral element is parabolic, then \(f_h\) is fixed-lift
proper at every puncture and its induced metric is complete on
\(\Sigma_0\).
\end{cor}

\begin{proof}
By Theorem~\ref{thm:mixed-conformal-correspondence}, the
no-common-zero condition makes \(f_h\) a conformal harmonic immersion.
The assertion at each puncture is
Theorem~\ref{thm:peripheral-classification}, and the final conclusion
is Corollary~\ref{cor:proper-and-complete-parabolic}.
\end{proof}

\section{Zero-weight nilpotent ends: unipotent monodromy and cusp geometry}
\label{sec:parabolic-ends}
This section analyses the local geometry near a puncture at which both the
principal parabolic weight and the semisimple residue parameter vanish, while
the nilpotent part of the Higgs-field residue is nonzero.
Throughout this section, all principal weights are assumed to lie in the
ordinary logarithmic range, and we use the normalised rank-three
representative of
Subsection~\ref{subsec:fixed-normalized-representative} and the global
initial metric of Definition~\ref{def:BGM-initial-model}.

Fix a puncture \(p\) and a \(0\)-polystable parabolic
\(\mathrm{PU}(2,\,1)\)-Higgs bundle \((\mathcal E_*,\Phi)\).  Because
\(\mathrm{PU}(2,\,1)\) is semisimple, the character-degree condition in
\cite[Theorem~5.1]{BGM2020} is vacuous.  Let \(h\) be the adapted
Hermitian--Einstein reduction supplied by that theorem relative to the fixed
global initial metric \(h_0\); in particular, \(h\) is quasi-isometric to
\(h_0\).  Let \(\rho\) be the holonomy
representation of the resulting flat connection, and let \(f\) be the
associated \(\rho\)-equivariant harmonic map, as in
Lemma~\ref{lem:sec6-geometric-dictionary}.
Assume that the principal local data of
Definition~\ref{def:principal-local-data} satisfy
\begin{equation}\label{eq:zero-weight-nilpotent-data}
\alpha_p\,=\,0,\qquad s_p\,=\,0,\qquad Y_p\,\ne\,0.
\end{equation}
Since \(\alpha_p\,=\,0\), the normalising cocharacter
\(\nu_p^{\mathrm{nor}}\) is zero.  Thus the compatible principal frame and
the normalised vector frame determine the same local logarithmic lattice.
The zero-weight projection defining the graded residue is the identity (see
\eqref{eq:GradedResiProjec}); after choosing the representative of its Levi
orbit, one therefore has
\begin{equation}\label{eq:residue-equals-Y}
N_p\,:=\,\operatorname{Res}_p(\Phi)\,=\,Y_p.
\end{equation}

\subsection{The residue type and the flat unipotent orbit}

We first classify the nilpotent residue in the normalised rank-three
representative and then identify the corresponding unipotent conjugacy class
of the peripheral holonomy.

\begin{prop}
\label{prop:nilpotent-end-type}
Under the assumption \eqref{eq:zero-weight-nilpotent-data}, the following
statements hold.

\begin{enumerate}[label=\textup{(\roman*)},leftmargin=2.8em]
\item Exactly one of the following two cases occurs:
\begin{enumerate}[label=\textup{(\Roman*)},leftmargin=2.8em]
\item \emph{Type I:} \(N_p^2\,=\,0\).

\item \emph{Type II:}
\(N_p^2\,\ne\,0\) and \(N_p^3\,=\,0\).
\end{enumerate}

\item The endomorphism \(N_p\) determines a canonical full flag
\begin{equation}\label{eq:nilpotent-residue-flag}
0\subset\mathcal F^N_{p,1}\subset\mathcal F^N_{p,2}\subset E_p.
\end{equation}
In both cases,
\(
N_p(\mathcal F^N_{p,2})\subset\mathcal F^N_{p,1}\) and \(
N_p(\mathcal F^N_{p,1})\,=\,0.
\)

\item Let \((H_p,X_p,Y_p)\) be the normalised Kostant--Rallis
\(\mathfrak{sl}_2\)-triple of \eqref{eq:PU21-KR-triple}.  Then
\begin{equation}\label{eq:flat-nilpotent-logarithm}
N_p^{\mathrm{flat}}
\,:=\,
\sqrt{-1}\,(Y_p-H_p-X_p)
\,=\,
\sqrt{-1}\,\operatorname{Ad}(e^{-X_p})Y_p
\end{equation}
is a nonzero nilpotent element of \(\mathfrak{pu}(2,\,1)\).  Under the
identification \(\mathfrak{pu}(2,\,1)\,\simeq\,\mathfrak{su}(2,\,1)\) and the
standard three-dimensional representation of the latter,
\(N_p^{\mathrm{flat}}\) has the same Jordan type as \(N_p\).

\item If \(c_p\) is a positively oriented peripheral loop about \(p\),
then the peripheral monodromy satisfies
\begin{equation}\label{eq:zero-weight-monodromy}
\rho(c_p)
\;\sim_G\;
\exp\bigl(2\pi\sqrt{-1}\,(Y_p-H_p-X_p)\bigr)
\,=\,
\exp\bigl(2\pi N_p^{\mathrm{flat}}\bigr).
\end{equation}
This is the full \(G\)-conjugacy class, rather than only its projection to a
Levi factor: the associated local-system weight is zero, so
\(P_0\,=\,L_0\,=\,G\) and the Levi projection is the identity.
In particular, \(\rho(c_p)\) is a nonidentity unipotent element of
\(\mathrm{PU}(2,\,1)\), and hence is parabolic.

\end{enumerate}
\end{prop}

\begin{proof}
By \eqref{eq:graded-residue-decomposition},
\eqref{eq:residue-equals-Y}, and
\eqref{eq:zero-weight-nilpotent-data}, \(N_p\) is a nonzero nilpotent
endomorphism of the three-dimensional vector space \(E_p\).  The only
partitions of \(3\) corresponding to nonzero nilpotent Jordan forms are
\[
(2,\,1)
\qquad\text{and}\qquad
(3).
\]
In the first case \(N_p^2\,=\,0\).  In the second case \(N_p^2\,\ne\,0\) and
\(N_p^3\,=\,0\).  This gives \textup{(i)}.

For the two types, define the residue flag by
\begin{equation}\label{eq:nilpotent-residue-flag-cases}
\begin{array}{@{}lcc@{}}
 &\mathcal F^N_{p,1}&\mathcal F^N_{p,2}\\
\textup{Type I}
 &\operatorname{im}N_p
 &\ker N_p\\[2pt]
\textup{Type II}
 &\ker N_p
 &\ker N_p^2\,=\,\operatorname{im}N_p.
\end{array}
\end{equation}
For Type~I, \(\operatorname{rank}N_p\,=\,1\), so
\(\dim\mathcal F^N_{p,1}\,=\,1\), \(\dim\mathcal F^N_{p,2}\,=\,2\), and
\(N_p(\mathcal F^N_{p,2})\,=\,0\subset\mathcal F^N_{p,1}\).  For Type~II,
the single Jordan block gives
\(\dim\ker N_p\,=\,1\), \(\dim\ker N_p^2\,=\,2\), and
\[
N_p(\mathcal F^N_{p,2})\,=\,\operatorname{im}N_p^2
\,=\,\ker N_p\,=\,\mathcal F^N_{p,1}.
\]
In both cases \(N_p(\mathcal F^N_{p,1})\,=\,0\), proving \textup{(ii)}.

As in Subsection~\ref{subsec:adapted-local-model}, let
\(\sigma\,=\,\theta\circ\tau_0\); its fixed-point algebra is
\(\mathfrak{su}(2,\,1)\,\simeq\,\mathfrak{pu}(2,\,1)\).  By
\eqref{eq:PU21-KR-triple},
\[
X_p\,=\,-\tau_0(Y_p),\qquad
H_p\in\sqrt{-1}\,\mathfrak h,\qquad
X_p,Y_p\in\mathfrak m^{\mathbb C}.
\]
The \(\mathfrak{sl}_2\)-relations in
\eqref{eq:PU21-KR-triple} or directly from \cite[Section~6.1, immediately after
equation~(6.1)]{BGM2020} give
\[
\operatorname{Ad}(e^{-X_p})Y_p\,=\,Y_p-H_p-X_p.
\]
Using the action of \(\theta\) on
\(\mathfrak h^{\mathbb C}\oplus\mathfrak m^{\mathbb C}\), together
with \eqref{eq:PU21-compact-conjugation}, we obtain
\[
\sigma\left(\operatorname{Ad}(e^{-X_p})Y_p\right)\,=\,-\operatorname{Ad}(e^{-X_p})Y_p.
\]
Since \(\sigma\) is conjugate-linear, it follows that
\[
\sigma\left(
\sqrt{-1}\,\operatorname{Ad}(e^{-X_p})Y_p
\right)\,=\,
\sqrt{-1}\,\operatorname{Ad}(e^{-X_p})Y_p.
\]
Therefore,
\[
\sqrt{-1}\,\operatorname{Ad}(e^{-X_p})Y_p
\in
\operatorname{Fix}(\sigma)
\,=\,
\mathfrak{su}(2,\,1)
\,\simeq\,
\mathfrak{pu}(2,\,1).
\]

Since \(\operatorname{Ad}(e^{-X_p})\) is an inner automorphism and
\(Y_p\,\ne\,0\), the element \(N_p^{\mathrm{flat}}\) is nonzero.
Conjugation preserves Jordan form, and multiplication by the nonzero
scalar \(\sqrt{-1}\) preserves the sizes of the nilpotent Jordan
blocks.  Therefore \(N_p^{\mathrm{flat}}\) has the same Jordan type as
\(Y_p\,=\,N_p\).  This gives \textup{(iii)}.

By \cite[Proposition~6.3]{BGM2020}, the weight of the associated
parabolic local system is
\[
\delta_p\,=\,s_p-\tau_0(s_p)\,=\,0.
\]
Since \(\delta_p\,=\,0\),
\[
P_{\delta_p}\,=\,P_0\,=\,G,
\qquad
L_{\delta_p}\,=\,Z_G(\delta_p)\,=\,Z_G(0)\,=\,G,
\]
and the projection \(P_{\delta_p}\longrightarrow  L_{\delta_p}\) is the identity.

With the positively oriented loop fixed in the statement and the compact
conjugation \(\tau_0(Z)\,=\,-Z^\dagger\) fixed in
\eqref{eq:PU21-compact-conjugation}, the displayed monodromy formula in
\cite[Proposition~6.3]{BGM2020} specialises at
\(\alpha_p\,=\,s_p\,=\,0\) to
\[
\rho(c_p)
\;\sim_G\;
\exp\!\left(2\pi\sqrt{-1}\,(Y_p-H_p-X_p)\right)
\,=\,\exp\bigl(2\pi N_p^{\mathrm{flat}}\bigr).
\]
The exponential of the nonzero nilpotent
\(2\pi N_p^{\mathrm{flat}}\) is a nonidentity
unipotent.  Its image in \(\mathrm{PU}(2,\,1)\) is also nonidentity,
since a central element of \(\mathrm{SU}(2,\,1)\) is unipotent only when
it is the identity.  Every nonidentity unipotent element of
\(\mathrm{PU}(2,\,1)\) is parabolic.  This gives \textup{(iv)}.
\end{proof}

\begin{remark}[Parabolic filtration versus residue flag]
\label{rem:residue-flag-not-parabolic-flag}
When \(\alpha_p\,=\,0\), one has \(P_{\alpha_p}\,=\,H^{\mathbb C}\); hence
\(Q_p\) is unique and the parabolic filtration of \(E_p\) is the trivial
filtration of weight zero.  By contrast, the nonzero nilpotent residue
\(N_p\) independently determines the canonical full flag
\eqref{eq:nilpotent-residue-flag-cases}.  This residue flag carries no
parabolic weights and is not additional parabolic data.
\end{remark}

\subsection{The energy--Higgs normalisation}
\label{subsec:energy-Higgs-normalization}

Let \(h\) be the adapted harmonic reduction fixed at the beginning of
the section.  By the normalisations chosen in
Sections~\ref{sec:Preliminaries} and~\ref{sec:parabolic-setup}, the
inner product induced by \(B|_{\mathfrak m}\) and the Bergman metric
\(g_{\mathbb{CH}}\) agree under the isotropy identification, up to a fixed positive
constant \(\kappa_{\mathrm B}\).

Under the identification fixed in the proof of
Lemma~\ref{lem:sec6-geometric-dictionary},
\[
f^*T\mathbb{CH}^2\,\simeq\,\varpi^*\mathcal E_h(\mathfrak m),
\]
equation~\eqref{eq:df-Phi-exact-normalization} reads
\[
(df)^{1,0}\,=\,c_0\,\varpi^*\Phi.
\]
Choose a local coordinate \(z\,=\,x+\sqrt{-1}y\).  We use
\(|dz|^2\,=\,dx^2+dy^2\) and the Hermitian extension of the Riemannian
cotangent norm; thus
\[
\|dz\|_{|dz|^2}^2\,=\,2.
\]
Write
\[
\Phi\,=\,\varphi\,dz.
\]
Thus \(f_z\) corresponds exactly to \(c_0\varphi\).  Since \(df\) is real,
\(f_{\bar z}\,=\,\overline{f_z}\), and
\[
f_x\,=\,f_z+f_{\bar z},
\qquad
f_y\,=\,\sqrt{-1}(f_z-f_{\bar z}).
\]
Therefore,
\begin{align*}
\frac12|df|_{|dz|^2,g_{\mathbb{CH}}}^2
&\,=\,\frac12\bigl(|f_x|_{g_{\mathbb{CH}}}^2+|f_y|_{g_{\mathbb{CH}}}^2\bigr)\\
&\,=\,2\,g_{\mathbb{CH}}^{\mathbb C}(f_z,\overline{f_z})\\
&\,=\,2\kappa_{\mathrm B}|c_0|^2|\varphi|_h^2.
\end{align*}
Our cotangent-norm convention gives
\[
\|\Phi\|_{h,|dz|^2}^2\,=\,2|\varphi|_h^2.
\]
Consequently, on setting
\(c_{\mathrm B}\,:=\,\kappa_{\mathrm B}|c_0|^2\,>\,0\), all factors coming from
the fixed differential--Higgs identification and the invariant-form
normalisation are absorbed into a single constant, and
\begin{equation}\label{eq:Phi-energy-normalization}
e_{|dz|^2}(f)
\,:=\,\frac12|df|_{|dz|^2,g_{\mathbb{CH}}}^2
\,=\,c_{\mathrm B}\|\Phi\|_{h,|dz|^2}^2.
\end{equation}

If \(f\) is weakly conformal, write
\(f^*g_{\mathbb{CH}}\,=\,\lambda|dz|^2\).  By
\eqref{eq:energy-conformal-coordinate}, its energy density is
\(\lambda\).  Therefore \eqref{eq:Phi-energy-normalization} gives
\begin{equation}\label{eq:Phi-metric-normalization}
f^*g_{\mathbb{CH}}
\,=\,c_{\mathrm B}\|\Phi\|_{h,|dz|^2}^2|dz|^2.
\end{equation}
The same identity holds at branch points, where both sides vanish.

\subsection{Two-sided cusp estimates}
We now combine the energy--Higgs normalisation with the adapted local model
to obtain upper and lower bounds of the same order near the puncture.

\begin{lemma}
\label{lem:local-cusp-norm}
Let \(h_{0,p}\) be the local initial metric of
Definition~\ref{def:BGM-initial-model}, and let \(h\) be a Hermitian
reduction on the punctured disc that is uniformly quasi-isometric to
\(h_{0,p}\).
Suppose that, at \(p\),
\[
\alpha_p\,=\,0,\qquad s_p\,=\,0,\qquad Y_p\,\ne\,0.
\]
Then, after shrinking the punctured coordinate disc
\(0\,<\,|z|\,<\,\varepsilon_0\,<\,1\), there exist constants
\(A_1,A_2\,>\,0\) such that
\begin{equation}\label{eq:Phi-cusp-norm}
\frac{A_1}{|z|\,\log(1/|z|)}
\;\le\;
\|\Phi\|_{h,|dz|^2}
\;\le\;
\frac{A_2}{|z|\,\log(1/|z|)}.
\end{equation}
\end{lemma}

\begin{proof}
Write \(h_0\,=\,h_{0,p}\) for this proof.  We first work with the local initial
metric \(h_0\).  In the
\(h_0\)-unitary frame used in
\cite[equations~(5.9) and~(5.12)]{BGM2020}, the assumptions
\(\alpha_p\,=\,s_p\,=\,0\) reduce the local expression for the Higgs field to
\begin{equation}\label{eq:BGM-zero-weight-Phi-expansion}
\Phi
\,=\,
\left(
-\frac{Y_p}{\log |z|^2}+R(z)
\right)\frac{dz}{z},
\qquad
\|R(z)\|_0\;\le\;C|z|^\varepsilon,
\end{equation}
for some constants \(C,\varepsilon\,>\,0\).  Here \(\|\cdot\|_0\) denotes the
coefficient norm in precisely this \(h_0\)-unitary frame.  In particular,
the remainder estimate in \cite[equation~(5.12)]{BGM2020} is measured in
this frame.

Since \(Y_p\,\ne\,0\), set
\[
\eta_p\,:=\,\|Y_p\|_0\,>\,0.
\]
Also,
\[
|z|^\varepsilon |\log |z|^2|\longrightarrow0
\qquad\text{as }z\longrightarrow0.
\]
After decreasing \(\varepsilon_0\), we may therefore assume that
\[
\|R(z)\|_0
\;\le\;
\frac{\eta_p}{2|\log |z|^2|}
\]
whenever \(0\,<\,|z|\,<\,\varepsilon_0\).  The reverse triangle inequality
applied to \eqref{eq:BGM-zero-weight-Phi-expansion} gives
\begin{align*}
\left\|
-\frac{Y_p}{\log |z|^2}+R(z)
\right\|_0
&\;\ge\;
\frac{\eta_p}{|\log |z|^2|}-\|R(z)\|_0\\
&\;\ge\;
\frac{\eta_p}{2|\log |z|^2|}.
\end{align*}
The ordinary triangle inequality gives
\begin{align*}
\left\|
-\frac{Y_p}{\log |z|^2}+R(z)
\right\|_0
&\;\le\;
\frac{\eta_p}{|\log |z|^2|}+\|R(z)\|_0\\
&\;\le\;
\frac{3\eta_p}{2|\log |z|^2|}.
\end{align*}
\[
\left\|\frac{dz}{z}\right\|_{|dz|^2}
\,=\,
\frac{\sqrt{2}}{|z|}.
\]
Thus
\[
\frac{\sqrt{2}\,\eta_p}{2|z|\,|\log |z|^2|}
\;\le\;
\|\Phi\|_{h_0,|dz|^2}
\;\le\;
\frac{3\sqrt{2}\,\eta_p}{2|z|\,|\log |z|^2|}.
\]
Since
\[
|\log |z|^2|\,=\,2\log(1/|z|),
\]
there exist constants \(a_1,a_2\,>\,0\) such that
\begin{equation}\label{eq:h0-Phi-cusp-comparison}
\frac{a_1}{|z|\,\log(1/|z|)}
\;\le\;
\|\Phi\|_{h_0,|dz|^2}
\;\le\;
\frac{a_2}{|z|\,\log(1/|z|)}.
\end{equation}

The assumed quasi-isometry \(h\,\simeq\,h_0\) is a uniformly bounded change of
reduction.  Since the isotropy representation is finite-dimensional, this
change and its inverse act with uniformly bounded operator norm.  Thus there
is a constant \(C_q\;\ge\;1\) such that
\[
C_q^{-1}\|v\|_{h_0}
\;\le\;
\|v\|_h
\;\le\;
C_q\|v\|_{h_0}
\]
for every local section \(v\) of the associated isotropy bundle.
Applying this comparison to \(\Phi\) and using
\eqref{eq:h0-Phi-cusp-comparison} proves
\eqref{eq:Phi-cusp-norm}, with
\[
A_1\,=\,C_q^{-1}a_1,
\qquad
A_2\,=\,C_q a_2.
\]
\end{proof}

\begin{prop}
\label{prop:cusp-asymptotic}
Under the hypotheses of Lemma~\ref{lem:local-cusp-norm}, suppose that
the associated harmonic map \(f\) is weakly conformal.  After shrinking
the disc sufficiently, \(df\) is nowhere zero and there are constants
\(c_1,c_2\,>\,0\) such that
\begin{equation}\label{eq:cusp-metric-comparison}
c_1\frac{|dz|^2}{|z|^2\log^2(1/|z|)}
\;\le\;f^*g_{\mathbb{CH}}\;\le\;
c_2\frac{|dz|^2}{|z|^2\log^2(1/|z|)}.
\end{equation}
These inequalities are inequalities of quadratic forms.
Moreover,
\begin{enumerate}
\item Every smaller cusp disc has finite energy.
\item Every piecewise \(C^1\) path tending to \(p\) has infinite induced
length.
\end{enumerate}

\end{prop}

\begin{proof}
The lower bound in \eqref{eq:Phi-cusp-norm} shows that \(\Phi\,\ne\,0\)
pointwise on the punctured disc.
The differential-vanishing criterion in
\eqref{eq:sec6-geometric-dictionary} therefore gives \(df\,\ne\,0\).  Squaring
\eqref{eq:Phi-cusp-norm} and using
\eqref{eq:Phi-metric-normalization} gives
\eqref{eq:cusp-metric-comparison}.

For \(0\,<\,\varepsilon\,<\,\varepsilon_0\), the energy identity
\eqref{eq:energy-conformal-coordinate} gives
\[
E\bigl(f;0\,<\,|z|\,<\,\varepsilon\bigr)
\;\le\;2\pi c_2\int_0^\varepsilon
\frac{dr}{r\log^2(1/r)}\,<\,\infty.
\]
If \(z:[0,1)\longrightarrow \{0\,<\,|z|\,<\,\varepsilon\}\) is a piecewise \(C^1\) path
tending to \(0\), put \(r(t)\,=\,|z(t)|\).  Since
\(|r'(t)|\;\le\;|z'(t)|\) almost everywhere, for \(0\,<\,T\,<\,1\),
\[
\begin{aligned}
L_f\bigl(z|_{[0,T]}\bigr)
&\;\ge\;\sqrt{c_1}\int_0^T
\frac{|r'(t)|}{r(t)\log(1/r(t))}\,dt\\
&\;\ge\;\sqrt{c_1}\left|
\log\log(1/r(T))-\log\log(1/r(0))
\right|.
\end{aligned}
\]
The right-hand side tends to \(+\infty\) as \(T\longrightarrow 1\).  This proves the
length assertion and, in particular, completeness towards the puncture.
\end{proof}

The preceding residue and metric analyses combine in the following local
cusp theorem.
\begin{theorem}
\label{thm:zero-weight-nilpotent-cusp}
Fix a puncture \(p\) and let
\((\mathcal E_*,\Phi)\) be a \(0\)-polystable parabolic
\(\mathrm{PU}(2,\,1)\)-Higgs bundle whose principal weights all lie in the
ordinary logarithmic range.  Let \(h\) be its adapted
Hermitian--Einstein reduction, quasi-isometric to the global initial metric
\(h_0\) of Definition~\ref{def:BGM-initial-model}; let \(\rho\) be the holonomy representation
of the corresponding flat connection, and let
\[
f:\mathbb H\longrightarrow\mathbb{CH}^2
\]
be the associated \(\rho\)-equivariant harmonic map.  Assume that the
principal local data at \(p\) satisfy
\[
\alpha_p\,=\,0,\qquad s_p\,=\,0,\qquad Y_p\,\ne\,0,
\]
and that \(f\) is weakly conformal.

Then the following statements hold:

\begin{enumerate}[label=\textup{(\roman*)},leftmargin=2.8em]
\item The peripheral monodromy \(\rho(c_p)\) is a nonidentity unipotent
element of \(\mathrm{PU}(2,\,1)\), and hence is parabolic.
\item For a coordinate \(z\) centred at \(p\), there are
\(c_1,c_2\,>\,0\) and \(\varepsilon\,>\,0\) such that
\[
c_1\frac{|dz|^2}{|z|^2\log^2(1/|z|)}
\;\le\;f^*g_{\mathbb{CH}}\;\le\;
c_2\frac{|dz|^2}{|z|^2\log^2(1/|z|)}
\qquad (0\,<\,|z|\,<\,\varepsilon).
\]
The inequalities are inequalities of quadratic forms.
In particular, \(df\) is nowhere zero on this cusp disc, and \(f\) has
finite energy there.  Every piecewise \(C^1\) path tending to \(p\) has
infinite induced length; hence the end is unbranched and complete.
\end{enumerate}
If these local assumptions hold at every puncture, the local conclusions
combine globally: \(f\) has
finite total energy, every peripheral monodromy element is
nonidentity unipotent, and \(f\) is fixed-lift proper at every
puncture.  Moreover, every divergent piecewise \(C^1\) path in
\(\Sigma_0\) has infinite \(f^*g_{\mathbb{CH}}\)-length.  Thus the
descended, possibly branched, conformal length metric is complete.  If \(f\)
is an immersion on all of \(\mathbb H\), it is a complete Riemannian metric.
\end{theorem}

\begin{proof}
Proposition~\ref{prop:nilpotent-end-type} gives \textup{(i)}.
Lemma~\ref{lem:local-cusp-norm} and
Proposition~\ref{prop:cusp-asymptotic} give \textup{(ii)}, including the
local finite-energy assertion.

If the assumptions hold at every puncture, choose disjoint cusp discs.
Their energies are finite by \textup{(ii)}, while the descended energy
density is smooth, and hence integrable, on the compact complement.  Thus
\(f\) has finite total energy.  Part~\textup{(i)} applies at every puncture,
and
Theorem~\ref{thm:peripheral-classification} then gives fixed-lift
properness at every end.  Choose a compact core whose complement is a
disjoint union of cusp discs.  A divergent path eventually remains in one
such disc and tends to its puncture, so part~\textup{(ii)} gives infinite
length.  If \(f\) is an immersion, equivariance makes
\(f^*g_{\mathbb{CH}}\) descend to a smooth positive-definite metric, and the
divergent-path criterion gives Riemannian completeness.
\end{proof}

Thus, near every puncture covered by the theorem,
\[
f^*g_{\mathbb{CH}}\,\asymp\,g_{\mathrm{cusp}},
\qquad
g_{\mathrm{cusp}}\,:=\,
\frac{|dz|^2}{|z|^2\log^2(1/|z|)},
\]
as quadratic forms: the induced metric is uniformly comparable to the
standard complete hyperbolic cusp metric.

\section{\texorpdfstring{Explicit Higgs data for $\rho$-equivariant $n$-noids with fixed-lift proper ends}{Explicit Higgs data for fixed-lift proper n-noids}}
\label{sec:n-noid}
Let \(P\subset\mathbb{CP}^1\) be a set of \(n\;\ge\;5\) distinct points.
We now give explicit zero-weight logarithmic Higgs data on
\((\mathbb{CP}^1,P)\).  The stability calculation singles out \(n\;\ge\;5\),
and nonabelian Hodge theory then produces the representation and conformal
minimal immersion.  Thus the construction is explicit at the level of Higgs
data; the harmonic metric, representation, and map are obtained through the
existence theorem.

\subsection{Set-up}

Fix \(n\;\ge\;5\) and distinct points
\(P\,=\,\{p_1,\ldots,p_n\}\subset\Sigma_c\,:=\,\mathbb{CP}^1\), and set
\[
\Sigma_0\,:=\,\Sigma_c\setminus P.
\]
Write
\(
K_{\log}\,:=\,K_{\Sigma_c}(P)\,\cong\,\mathcal O_{\mathbb{CP}^1}(n-2).
\)
Fix an isomorphism \(K_{\log}\,\simeq\,\mathcal O(n-2)\); the component
formulas below are written relative to this choice.  Since \(n\;\ge\;5\), the
universal cover of \(\Sigma_0\) is \(\mathbb H\) (indeed, \(n\;\ge\;3\)
suffices for this uniformisation).  The stronger threshold \(n\;\ge\;5\)
enters only in the stability calculation.

\begin{definition}[{$\mathbb{CH}^2$-$n$-noid}]\label{def:CH2-nnoid}
A \(\mathbb{CH}^2\)-\(n\)-noid is a
\(\rho\)-equivariant conformal minimal immersion
\[
f:\mathbb H\longrightarrow\mathbb{CH}^2,
\qquad
\rho:\pi_1(\Sigma_0)\longrightarrow\mathrm{PU}(2,\,1)
\]
for a reductive representation \(\rho\), such that the
\(\pi_1(\Sigma_0)\)-invariant metric \(f^*g_{\mathbb{CH}}\) descends
to a complete Riemannian metric on \(\Sigma_0\).

This terminology is nonstandard: neither \(f\) itself descends to
\(\Sigma_0\) nor is global properness of
\(f:\mathbb H\longrightarrow \mathbb{CH}^2\) required.  In the examples below, finite
energy and nonidentity unipotent peripheral holonomy imply fixed-lift
properness at every puncture.
\end{definition}

\subsection{Explicit Higgs data}

Let
\[
V_0\,:=\,\mathcal O(1)\oplus\mathcal O,\qquad
L_0\,:=\,\mathcal O(-1),\qquad
E_0\,:=\,V_0\oplus L_0.
\]
Thus \(\det E_0\,\simeq\,\mathcal O\).  Equip \(E_0\) at every
\(p_i\) with the zero vector-weight parabolic structure, so all
three vector-bundle weights are zero.  We seek an off-diagonal
logarithmic Higgs field
\[
\Phi\,=\,\begin{pmatrix}0&\Phi_1\\ \Phi_2&0\end{pmatrix},
\]
where
\(\Phi_1:L_0\longrightarrow V_0\otimes K_{\log}\) and
\(\Phi_2:V_0\longrightarrow L_0\otimes K_{\log}\).
Using the splitting of \(V_0\), we obtain
\[
\operatorname{Hom}(V_0,L_0)\otimes K_{\log}
\,\simeq\,\mathcal O(n-4)\oplus\mathcal O(n-3),
\]
\[
\operatorname{Hom}(L_0,V_0)\otimes K_{\log}
\,\simeq\,\mathcal O(n)\oplus\mathcal O(n-1).
\]
Choose sections
\[
g_1\in H^0(\mathcal O(n-4)),\qquad
g_2\in H^0(\mathcal O(n-3)),\qquad
q\in H^0(\mathcal O(3))
\]
such that \(g_1\) and \(g_2\) have no common zero and
\(q(p_i)\,\ne\,0\) for every \(i\).
Such sections always exist.  In homogeneous coordinates \([X:Y]\), one
may take \(g_1\,=\,X^{n-4}\) and \(g_2\,=\,Y^{n-3}\); their zero sets are disjoint.
One may then choose a cubic \(q\) whose zeros avoid the finite set \(P\).

Define
\begin{equation}\label{eq:Higgs-data}
\Phi_2\,=\,(g_1,g_2),
\qquad
\Phi_1\,=\,\begin{pmatrix}-qg_2\\ qg_1\end{pmatrix},
\qquad
\Phi\,=\,\begin{pmatrix}0&\Phi_1\\ \Phi_2&0\end{pmatrix}.
\end{equation}
The products \(qg_2\in H^0(\mathcal O(n))\) and
\(qg_1\in H^0(\mathcal O(n-1))\) have the required component degrees, so
\eqref{eq:Higgs-data} is globally defined.  With this zero vector-weight
parabolic structure, the projectivisation of \((E_0,\Phi)\) is a parabolic
\(\mathrm{PU}(2,\,1)\)-Higgs bundle with principal weights
\(\alpha_{p_i}\,=\,0\).

To use the normalised rank-three representative of
Subsection~\ref{subsec:fixed-normalized-representative}, tensor by
\(L_0^{-1}\,=\,\mathcal O(1)\):
\[
E'\,:=\,E_0\otimes\mathcal O(1)
\,=\,\bigl(\mathcal O(2)\oplus\mathcal O(1)\bigr)\oplus\mathcal O.
\]
Under the natural Hom-bundle identifications, \(\Phi_1\) becomes a map
\(\mathcal O\longrightarrow 
(\mathcal O(2)\oplus\mathcal O(1))\otimes K_{\log}\), while
\(\Phi_2\) becomes a map
\(\mathcal O(2)\oplus\mathcal O(1)\longrightarrow  K_{\log}\).
The common twist changes neither their zero loci nor the projectivised
Higgs bundle.  Since \(\Phi_2\) is nowhere zero and \(\Phi_1\) vanishes
exactly where \(q\) does, the divisors defined in
Definitions~\ref{def:bundle-valued-zero-divisor}
and~\ref{def:line-valued-zero-divisor} are
\begin{equation}\label{eq:nnoid-divisors}
D_1\,=\,\operatorname{div}(\Phi_1)\,=\,\operatorname{div}(q),\qquad
D_2\,=\,\operatorname{div}(\Phi_2)\,=\,0,
\qquad
(d_1,d_2)\,=\,(3,0).
\end{equation}
In particular, the two Higgs components have no common zero.

\subsection{Conformality and nilpotent residues}

The defining sections satisfy
\begin{equation}\label{eq:nnoid-syzygy}
\Phi_2\Phi_1
\,=\,(g_1,g_2)
\begin{pmatrix}-qg_2\\qg_1\end{pmatrix}
\,=\,0.
\end{equation}
In the normalised rank-three representative,
\eqref{eq:trace-Phi-square} gives
\(\operatorname{tr}(\Phi^2)\,=\,2\Phi_2\Phi_1\,=\,0\).

\begin{lemma}[Type~II residues of the explicit data]
\label{lem:nilpotent-residues}
At each \(p_i\), the ordinary and graded residues coincide, and the
endomorphism
\[
N_i\,:=\,\operatorname{Res}_{p_i}(\Phi)
\]
satisfies \(N_i^3\,=\,0\) and \(N_i^2\,\ne\,0\).  Thus \(N_i\) has Jordan type
\((3)\).  The principal local data are
\[
\alpha_{p_i}\,=\,s_{p_i}\,=\,0,\qquad Y_{p_i}\,=\,N_i\,\ne\,0.
\]
\end{lemma}

\begin{proof}
Near \(p_i\), choose compatible holomorphic frames of the line-bundle
summands of \(V_0\oplus L_0\), a local frame \(e_q\) of
\(\mathcal O(3)\), and a coordinate \(z\) with \(z(p_i)\,=\,0\).  Choose these
frames to respect the tensor identifications used in
\eqref{eq:Higgs-data}.  Since
\(dz/z\) is a local frame of \(K_{\log}\), the coefficient row of
\(\Phi_2\) and the section \(q\) have expansions
\[
\Phi_2\,=\,
\bigl(a_i+O(z),b_i+O(z)\bigr)\frac{dz}{z},
\qquad
q\,=\,\bigl(q_i+O(z)\bigr)e_q.
\]
The hypotheses give \((a_i,b_i)\,\ne\,(0,0)\) and \(q_i\,\ne\,0\).  Therefore
\[
\operatorname{Res}_{p_i}(\Phi)\,=\,
\begin{pmatrix}
0 & B_i\\
C_i & 0
\end{pmatrix}
\,=\,N_i,
\qquad
B_i\,=\,\binom{-q_i b_i}{q_i a_i},\quad C_i\,=\,(a_i,b_i).
\]
Here \(C_iB_i\,=\,0\), whereas \(B_i\,\ne\,0\) and \(C_i\,\ne\,0\).  Hence
\(N_i^3\,=\,0\), and
\[
N_i^2\,=\,\begin{pmatrix}B_iC_i&0\\0&0\end{pmatrix},
\]
with \(B_iC_i\,\ne\,0\): choose indices \(r,s\) for which
\((B_i)_r\,\ne\,0\) and \((C_i)_s\,\ne\,0\); then the \((r,s)\)-entry is
nonzero.  Thus \(N_i^2\,\ne\,0\), proving the Jordan-type assertion.

Because \(\alpha_{p_i}\,=\,0\), the zero-weight projection in
Definition~\ref{def:principal-local-data} is the identity; thus the
graded residue is the ordinary residue \(N_i\), as observed at the
beginning of Section~\ref{sec:parabolic-ends}.  Since \(N_i\) is
nilpotent, its Jordan decomposition has semisimple part zero and
nilpotent part \(N_i\).  Hence
\[
s_{p_i}\,=\,0,
\qquad
Y_{p_i}\,=\,N_i\,\ne\,0.
\]
Thus the residue is of Type~II in the terminology of
Proposition~\ref{prop:nilpotent-end-type}.
\end{proof}

\subsection{Stability}

\begin{prop}\label{prop:nnoid-stability}
For every \(n\;\ge\;5\), the zero vector-weight Higgs data
\eqref{eq:Higgs-data} define a stable parabolic
\(\mathrm{PU}(2,\,1)\)-Higgs bundle.
\end{prop}
\begin{proof}
By \eqref{eq:nnoid-divisors},
\[
d_1\,=\,3,\qquad d_2\,=\,0,\qquad
\kappa\,=\,\deg K_{\log}\,=\,n-2.
\]
Since all normalised vector-bundle weights vanish,
Proposition~\ref{prop:sec7-explicit-stability} gives the two conditions
\[
2d_1+d_2\,<\,3\kappa,
\qquad
d_1+2d_2\,<\,3\kappa.
\]
For the present data these become
\[
6\,<\,3(n-2),
\qquad
3\,<\,3(n-2).
\]
The first inequality is equivalent to \(n\,>\,4\), whereas the second is
equivalent to \(n\,>\,3\).  Thus both inequalities hold simultaneously
exactly when \(n\;\ge\;5\).  Proposition~\ref{prop:sec7-explicit-stability}
therefore proves stability of the
parabolic \(\mathrm{PU}(2,\,1)\)-Higgs bundle.
\end{proof}

\begin{theorem}
\label{thm:CH2-n-noid}
Let \(n\;\ge\;5\), let
\(P\,=\,\{p_1,\ldots,p_n\}\subset\mathbb{CP}^1\) consist of distinct
points, and fix an isomorphism
\(K_{\mathbb{CP}^1}(P)\,\simeq\,\mathcal O(n-2)\).  Choose
\[
\begin{aligned}
g_1&\in H^0(\mathcal O(n-4)),&
g_2&\in H^0(\mathcal O(n-3)),&
q&\in H^0(\mathcal O(3)),
\end{aligned}
\]
where \(g_1,g_2\) have no common zero and \(q(p_i)\,\ne\,0\) for every
\(i\).  Then the logarithmic Higgs bundle
\eqref{eq:Higgs-data}, with zero principal and vector-bundle weights,
determines a reductive representation, well defined up to conjugacy,
\[
\rho:\pi_1(\mathbb{CP}^1\setminus P)\longrightarrow\mathrm{PU}(2,\,1)
\]
and an associated mixed, unbranched, \(\rho\)-equivariant conformal
minimal immersion
\[
f:\mathbb H\longrightarrow\mathbb{CH}^2.
\]
Every peripheral monodromy element is nonidentity unipotent.  The map has
finite total energy and is fixed-lift proper at every puncture.
The invariant metric \(f^*g_{\mathbb{CH}}\) descends to a complete Riemannian metric on
\(\mathbb{CP}^1\setminus P\).  Thus \(f\) is a
\(\mathbb{CH}^2\)-\(n\)-noid in the sense of
Definition~\ref{def:CH2-nnoid}.
\end{theorem}

\begin{proof}
Set \(\Sigma_0\,=\,\mathbb{CP}^1\setminus P\).  The syzygy
\eqref{eq:nnoid-syzygy}, the trace identity
\eqref{eq:trace-Phi-square}, the divisor calculation
\eqref{eq:nnoid-divisors}, and Proposition~\ref{prop:nnoid-stability}
verify the hypotheses of
Theorem~\ref{thm:mixed-conformal-correspondence},
part~\ref{item:higgs-side}.  That theorem yields an adapted harmonic
reduction \(h\), a reductive representation \(\rho\), and an associated
mixed, unbranched, \(\rho\)-equivariant conformal minimal immersion
\(f\).

Lemma~\ref{lem:nilpotent-residues} gives zero-weight Type~II local data
at every puncture.  Because \(h\) and \(\rho\) have now been constructed,
part~\textup{(iv)} of Proposition~\ref{prop:nilpotent-end-type} shows
that every peripheral monodromy element is nonidentity unipotent.

Theorem~\ref{thm:zero-weight-nilpotent-cusp} yields finite total energy
and fixed-lift properness at every puncture.  Since every peripheral
element is parabolic, Corollary~\ref{cor:proper-and-complete-parabolic}
shows that \(f^*g_{\mathbb{CH}}\) descends to a complete metric on \(\Sigma_0\).
Definition~\ref{def:CH2-nnoid} now completes the proof.
\end{proof}
\appendix 

\section{Details for the parabolic
\texorpdfstring{$\mathrm{PU}(2,\,1)$}{PU(2,\,1)} framework}
\label{app:parabolic-details}

This appendix expands the material summarised in
Section~\ref{sec:parabolic-setup}.  It contains the calculations needed to
pass from the local data of the principal bundle to the normalised
rank-three representative.  The general
theory is developed in~\cite{BGM2020}.
The details are included because the residue, divisor, stability, and cusp
calculations require an explicit form of the general theory.

\subsection{Local weights, parabolic fields, and graded residues}
\label{app:local-data-calculation}
Recall that \(\Sigma_0\,=\,\Sigma\setminus P\),
\(K_{\log}\,=\,K_\Sigma(P)\),
\(G\,=\,\mathrm{PU}(2,\,1)\), and
\(H^{\mathbb C}\,\simeq\,\mathrm{GL}_2(\mathbb C)\).  For the Cartan
decomposition
\[
\mathfrak g^{\mathbb C}
\,=\,\mathfrak h^{\mathbb C}
\oplus
\mathfrak m^{\mathbb C},
\]
the subspace \(\mathfrak h^{\mathbb C}\) is represented by block-diagonal
matrices, whereas \(\mathfrak m^{\mathbb C}\) is represented by
block-off-diagonal matrices relative to the decomposition
\(
\mathbb C^3\,=\,\mathbb C^2\oplus\mathbb C
\).
Thus, in the present case,
\[
\mathfrak m^{\mathbb C}
\,=\,
\left\{\begin{pmatrix}0&x\\ \xi&0\end{pmatrix}\;\big\vert\;
x\,\in\,\mathbb C^2,\ \xi\,\in\,(\mathbb C^2)^\vee\right\},
\ \ \,
a\cdot(x,\,\xi)\,=\,(ax,\,\xi a^{-1}).
\]

An element of \(\mathfrak h^{\mathbb C}\) may be represented by a
\(2\times2\) matrix \(A\,\in\,\mathfrak{gl}_2(\mathbb C)\).
On the other hand,
\[
\mathfrak h^{\mathbb C}\ \subset\
\mathfrak g^{\mathbb C}\,=\,\
\mathfrak{sl}_3(\mathbb C),
\]
so when \(\mathfrak h^{\mathbb C}\) is viewed inside
\(\mathfrak{sl}_3(\mathbb C)\), its elements are represented by
block-diagonal \(3\times3\) traceless matrices.

At \(p\,\in\, P\), the reduction
\(Q_p\) and the weight
\(\alpha_p\in\sqrt{-1}\,\overline{\mathcal A}\) constitute the parabolic
structure in the sense of~\cite[Definitions~2.1--2.2]{BGM2020}.  A
\emph{compatible frame} is a local
holomorphic trivialisation whose value at \(p\) represents \(Q_p\), so that
the corresponding subgroup of \(H^{\mathbb C}\) is \(P_{\alpha_p}\).

Choose the diagonal maximal torus in \(H\,\simeq\,\mathrm U(2)\) and write
\[
\alpha_p\,=\,\operatorname{diag}(u_p,v_p).
\]
The Weyl group interchanges the two eigenlines.  Let
\(\mathcal A^f\subset\overline{\mathcal A}\) be the chosen fundamental
domain for the affine Weyl-group action.  It selects the affine-Weyl
representative, while the chosen Weyl chamber orders unequal eigenvalues.
On a wall with repeated eigenvalues, no individual eigenline is intrinsically
selected; compare \cite[Appendix~A.1, Propositions~A.1--A.2]{BGM2020}.

The principal weight above is a \(2\times2\) element of
\(\mathfrak h^{\mathbb C}\).  The split rank-three vector lift
uses the lift
\[
A\longmapsto\operatorname{diag}(A,0),
\qquad
a\longmapsto\operatorname{diag}(a,1),
\]
whereas ambient calculations take place in
\(\mathfrak g^{\mathbb C}\,=\,\mathfrak{sl}_3(\mathbb C)\).  For those
calculations, we use the traceless lift
\begin{equation}\label{eq:app-trace-free-lift}
\jmath(A)\,:=\,\operatorname{diag}(A,0)
-\frac{\operatorname{tr}(A)}{3}I_3.
\end{equation}
Thus \(\jmath(A)\in\mathfrak{sl}_3(\mathbb C)\), and for every
\(X\in\mathfrak{gl}_3(\mathbb C)\),
\[
[\jmath(A),X]\,=\,[\operatorname{diag}(A,0),X].
\]
For the general Cartan decomposition and isotropy representation, see
Appendix~A.2 of \cite{BGM2020}.  The explicit map \(\jmath\) above is our
matrix realisation.

For
\[
A\,=\,\begin{pmatrix}a&b\\ c&d\end{pmatrix}
\in\mathfrak h^{\mathbb C}\,\simeq\,\mathfrak{gl}_2(\mathbb C),
\]
one has
\[
[\alpha_p,A]
\,=\,
\begin{pmatrix}
0&(u_p-v_p)b\\
(v_p-u_p)c&0
\end{pmatrix}.
\]
For this isotropy representation, write
\(x\,=\,(x_1,x_2)^t\) and \(\xi\,=\,(\xi_1,\xi_2)\).  Differentiating the action
gives
\begin{equation}\label{eq:app-ad-alpha-components}
\operatorname{ad}(\jmath(\alpha_p))(x_1,x_2,\xi_1,\xi_2)
\,=\,
(u_px_1,v_px_2,-u_p\xi_1,-v_p\xi_2).
\end{equation}
Thus the possible nonzero eigenvalues of
\(\operatorname{ad}(\jmath(\alpha_p))\) on
\(\mathfrak g^{\mathbb C}\) are
\[
\pm u_p,\qquad \pm v_p,\qquad \pm(u_p-v_p),
\]
which proves the equivalence asserted in
Definition~\ref{def:ordinary-range}.  More precisely, our term
``ordinary logarithmic range'' denotes the condition defining
\(\mathcal A'_{\mathfrak g}\) in~\cite[equation~(4.3)]{BGM2020}.

The bounded-conjugation description of a parabolic subgroup and its Levi
factor is recalled in~\cite[Appendix~B.1]{BGM2020}.  In our model, it gives
\begin{equation}\label{eq:app-explicit-parabolic}
P_{\alpha_p}
\,=\,
\begin{cases}
\left\{
\begin{pmatrix}a&0\\ c&d\end{pmatrix}:
a,d\in\mathbb C^*,\ c\in\mathbb C
\right\},&u_p\,>\,v_p,\\[3mm]
\left\{
\begin{pmatrix}a&b\\ 0&d\end{pmatrix}:
a,d\in\mathbb C^*,\ b\in\mathbb C
\right\},&u_p\,<\,v_p,\\[3mm]
\mathrm{GL}_2(\mathbb C),&u_p\,=\,v_p,
\end{cases}
\end{equation}
\begin{equation}\label{eq:app-explicit-Levi}
L_{\alpha_p}
\,=\,
Z_{\mathrm{GL}_2(\mathbb C)}(\alpha_p)
\,=\,
\begin{cases}
\{\operatorname{diag}(a,d):a,d\in\mathbb C^*\},&u_p\,\ne\,v_p,\\
\mathrm{GL}_2(\mathbb C),&u_p\,=\,v_p.
\end{cases}
\end{equation}
Similarly, \eqref{eq:app-ad-alpha-components} gives
\begin{equation}\label{eq:app-explicit-m-alpha}
\mathfrak m_{\alpha_p}
\,=\,
\left\{
(x_1,x_2,\xi_1,\xi_2):
\begin{array}{ll}
x_1\,=\,0&\text{if }u_p\,>\,0,\\
x_2\,=\,0&\text{if }v_p\,>\,0,\\
\xi_1\,=\,0&\text{if }u_p\,<\,0,\\
\xi_2\,=\,0&\text{if }v_p\,<\,0
\end{array}
\right\},
\end{equation}
whereas its zero-weight part is
\begin{equation}\label{eq:app-explicit-m-alpha-zero}
\mathfrak m^0_{\alpha_p}
\,=\,
\left\{(x_1,x_2,\xi_1,\xi_2):
\begin{array}{ll}
x_1\,=\,\xi_1\,=\,0&\text{if }u_p\,\ne\,0,\\
x_2\,=\,\xi_2\,=\,0&\text{if }v_p\,\ne\,0
\end{array}
\right\}.
\end{equation}
Let
\(
\mathfrak n_{\alpha_p}\,:=\,
\bigoplus_{\mu\,<\,0}\mathfrak m^{\mathbb C}_\mu
\), so that
\(\mathfrak m_{\alpha_p}\,=\,\mathfrak m^0_{\alpha_p}
\oplus\mathfrak n_{\alpha_p}\).

Choose a local frame near \(p\) that is compatible with the parabolic data
\((\alpha_p,Q_p)\).  In this frame, a local section of
the isotropy bundle is represented by a meromorphic function
\[
\varphi\,=\,\sum_\mu \varphi_\mu,
\]
where \(\varphi_\mu\) denotes the component lying in the
\(\mu\)-eigenspace of
\(
\operatorname{ad}(\jmath(\alpha_p)).
\)
In \cite{BGM2020} the authors define the parabolic isotropy sheaf by requiring each eigenspace
component to have controlled pole order at \(p\).  To distinguish the
valuation from the weight \(v_p\), write it as \(\operatorname{ord}_p\):
\begin{equation}\label{eq:app-BGM-order-condition}
\operatorname{ord}_p(\varphi_\mu)\;\ge\;-\lfloor-\mu\rfloor.
\end{equation}
Here \(\operatorname{ord}_p(\varphi_\mu)\) denotes the order of vanishing of
\(\varphi_\mu\) at \(p\), with negative values corresponding to poles.
They define the strictly parabolic sheaf
\(\mathcal N\mathcal E(\mathfrak m^{\mathbb C})\) by the corresponding
strict condition.  The invariance of both sheaves under compatible changes
of frame is established in~\cite[equation~(4.1), Lemma~4.1, and
Remark~4.2]{BGM2020}.  In the ordinary range, the intrinsic description
immediately following \cite[equation~(4.3)]{BGM2020} becomes
\[
\begin{aligned}
\mathcal P\mathcal E(\mathfrak m^{\mathbb C})
&\,=\,\{\varphi\text{ holomorphic}:\varphi(p)\in
\mathfrak m_{\alpha_p}\},\\
\mathcal N\mathcal E(\mathfrak m^{\mathbb C})
&\,=\,\{\varphi\text{ holomorphic}:\varphi(p)\in
\mathfrak n_{\alpha_p}\}.
\end{aligned}
\]
Indeed, \eqref{eq:app-BGM-order-condition} says that a positive-weight
coefficient vanishes at \(p\), while a nonpositive-weight coefficient is
holomorphic without a forced zero.

Choose a coordinate \(z\) centred at \(p\) and a frame compatible with
the reduction \(Q_p\).  Write
\[
\Phi\,=\,(\beta(z),\gamma(z))\frac{dz}{z},
\qquad
\beta\,=\,\begin{pmatrix}\beta_1\\ \beta_2\end{pmatrix},
\qquad
\gamma\,=\,(\gamma_1,\gamma_2).
\]
In the ordinary logarithmic range, define
\[
\begin{aligned}
V_{p,w}
&\,:=\,\ker(\alpha_p-wI)\subset\mathbb C^2,\\
V_{p,w}^{\vee}
&\,:=\,\left\{\xi\in(\mathbb C^2)^\vee:
\xi(V_{p,w'})\,=\,0\ \text{for every }w'\,\ne\,w\right\}.
\end{aligned}
\]
The order condition in \cite{BGM2020} reduces to
\begin{equation}\label{eq:app-vector-boundary-condition}
\beta(0)\in
\bigoplus_{\substack{w\in\{u_p,v_p\}\\ w\;\le\;0}}V_{p,w},
\qquad
\gamma(0)\in
\bigoplus_{\substack{w\in\{u_p,v_p\}\\ w\;\ge\;0}}V_{p,w}^{\vee}.
\end{equation}

As explained in Section~\ref{sec:parabolic-setup}, in a centred local
coordinate \(z\) and a frame compatible with the parabolic reduction, the
Higgs field has the form
\[
\Phi\,=\,\bigl(R_p+O(z)\bigr)\frac{dz}{z},
\]
where \(R_p\) is a representative of the residue.

We define \[
\operatorname{Gr}\operatorname{Res}_p(\Phi)
\,=\,
\operatorname{pr}_{\mathfrak m^0_{\alpha_p}}(R_p),
\]
which is exactly the definition given in
\eqref{eq:GradedResiProjec}.  The individual element
\(\operatorname{Gr}\operatorname{Res}_p(\Phi)\) depends on the choice
of compatible frame, but its \(L_{\alpha_p}\)-orbit does not.  The
coordinate and gauge calculations establishing this invariance are carried
out in~\cite[Section~4.1, equations~(4.4)--(4.5) and the discussion following
them]{BGM2020}.  After choosing a representative of the intrinsic
\(L_{\alpha_p}\)-orbit, the graded residue has its commuting Jordan
decomposition
\[
\operatorname{Gr}\operatorname{Res}_p(\Phi)
\,=\,
s_p+Y_p,
\qquad
[s_p,Y_p]\,=\,0,
\]
where \(s_p\) is semisimple and \(Y_p\) is nilpotent;
compare~\cite[equation~(5.2)]{BGM2020}.

Thus the intrinsic local information at the puncture is
\[
\bigl(
\alpha_p,Q_p;
L_{\alpha_p}\!\cdot(s_p+Y_p)
\bigr).
\]
Here \(\alpha_p\) gives the parabolic weight, \(Q_p\) gives the
parabolic reduction, and
\(L_{\alpha_p}\!\cdot(s_p+Y_p)\) gives the frame-independent orbit of
the graded residue.  Only after \(Q_p\) has been understood and a specific
representative \(s_p+Y_p\) of this orbit has been chosen do we use the
shorter notation
\[
(\alpha_p,s_p,Y_p).
\]


\subsection{Hecke normalisation of the associated vector bundle}\label{app:normalization-details}

We keep the parabolic principal \(H^{\mathbb C}\)-bundle
\(\mathcal E_*\) fixed throughout.  On the punctured surface
\(\Sigma_0\,=\,\Sigma\setminus P\), its standard representation determines the
rank-two holomorphic vector bundle
\[
V_0
\,=\,
\mathcal E_0\times_{\mathrm{GL}_2(\mathbb C)}\mathbb C^2,
\qquad
\mathcal E_0\,=\,\mathcal E|_{\Sigma_0}.
\]
The parabolic structure on \(\mathcal E_*\) determines a holomorphic
extension of \(V_0\) across each puncture;  we denote this initial
extension by \(V^{\mathrm{raw}}\).  The purpose of the Hecke transformation
below is to replace this raw extension by a normalised holomorphic lattice
of the same bundle \(V_0\) on \(\Sigma_0\), with vector-bundle weights in
the interval \([0,1)\).

Fix \(p\in P\), choose a centred coordinate \(z\) on a disc
\(\Delta_p\), and choose a compatible diagonal frame
\[
e_p\,=\,(e_{p,1},e_{p,2})
\]
of \(V^{\mathrm{raw}}\).  Write
\[
w_{p,1}\,=\,u_p,\qquad w_{p,2}\,=\,v_p
\]
for the corresponding weights in the standard
\(\mathrm{GL}_2(\mathbb C)\)-representation.  An integral cocharacter
\[
\nu_p\,=\,\operatorname{diag}(k_{p,1},k_{p,2})
\in X_*(\mathrm{GL}_2)\,\simeq\,\mathbb Z^2
\]
acts on the local frame by
\[
h_p(z)\,:=\,z^{\nu_p},
\qquad
e_p\longmapsto e_p h_p^{-1}.
\]
Equivalently, it replaces the local lattice
\[
\mathcal O_{\Delta_p}e_{p,1}
\oplus
\mathcal O_{\Delta_p}e_{p,2}
\]
by
\begin{equation}\label{eq:app-normalized-lattice}
V|_{\Delta_p}\,:=\,
\mathcal O_{\Delta_p}z^{-k_{p,1}}e_{p,1}
\oplus
\mathcal O_{\Delta_p}z^{-k_{p,2}}e_{p,2}.
\end{equation}
Both lattices restrict to the same holomorphic bundle \(V_0\) on the
punctured disc \(\Delta_p^*\); what changes is only the extension across
\(p\).  This is the vector-bundle form of the Hecke modification associated
with the cocharacter \(\nu_p\).  At the same time the corresponding weights
are shifted by \(\nu_p\); see
\cite[Section~3.3, Lemma~3.3, and Theorem~3.4]{BGM2020}.
Our convention
\(X_*(\mathrm{GL}_2)\,\simeq\,\mathbb Z^2\) absorbs the
\(\sqrt{-1}/(2\pi)\)-normalisation for cocharacters used there.

For the normalisation used in this article, \(\nu_p\) is chosen so that
the resulting increasing vector-bundle weights lie in \([0,1)\).
Since \(w_{p,j}\in(-1,1)\), set
\[
\begin{aligned}
\nu_p^{\mathrm{nor}}
  &\,:=\,\operatorname{diag}(k_{p,1},k_{p,2}),
  & k_{p,j}&\,:=\,-\lceil w_{p,j}\rceil,\\
a_{p,j}
  &\,:=\,-(w_{p,j}+k_{p,j})
  &\,=\,\lceil w_{p,j}\rceil-w_{p,j}\in[0,1).
\end{aligned}
\]
Thus \(a_{p,j}\) is the transported weight after conversion from the
decreasing-weight convention in \cite{BGM2020} to the increasing vector-bundle convention
used in this article.  The sign is fixed by
\cite[Remark~2.6 and equation~(2.4)]{BGM2020}.

If \(z'\,=\,u(z)z\), with \(u(0)\,\ne\,0\), the two normalised frames differ by
the holomorphic invertible gauge
\[
u(z)^{-\nu_p^{\mathrm{nor}}},
\]
so the isomorphism class of \eqref{eq:app-normalized-lattice} is independent
of the centred coordinate.

Let \(L_{p,j}\subset V_p\) be the line represented by the \(j\)-th
normalised frame vector.  Translating the decreasing flag convention
of \cite[Example~2.5 and Remark~2.6]{BGM2020} into our increasing
convention gives the weighted flag
\[
\begin{array}{@{}lll@{}}
\toprule
\text{relation between the weights}&\text{flag}&\text{successive weights}\\
\midrule
a_{p,1}\,<\,a_{p,2}&V_p\supset L_{p,2}\supset0&a_{p,1}\,<\,a_{p,2}\\
a_{p,2}\,<\,a_{p,1}&V_p\supset L_{p,1}\supset0&a_{p,2}\,<\,a_{p,1}\\
a_{p,1}\,=\,a_{p,2}&V_p\supset0&a_{p,1}.\\
\bottomrule
\end{array}
\]
When \(a_{p,1}\,=\,a_{p,2}\), neither individual eigenline enters the
parabolic flag, so the resulting trivial flag \(V_p\supset0\) is independent
of the chosen diagonal frame.

Together with the local lattices \eqref{eq:app-normalized-lattice}, these
flags define the normalised parabolic vector bundle \(V_*\).  Thus the
construction is
\[
\mathcal E_*
\longmapsto
V_0
\longmapsto
V^{\mathrm{raw}}
\xrightarrow{\;\nu^{\mathrm{nor}}\;}
V
\longmapsto
V_*.
\]

The Higgs field is transported through the same Hecke gauge.  If
\[
\Phi\,=\,(\beta,\gamma)\,\frac{dz}{z}
\]
in the raw frame, then
\begin{equation}\label{eq:app-Hecke-Higgs-transform}
\beta^{\mathrm{nor}}\,=\,h_p\beta,
\qquad
\gamma^{\mathrm{nor}}\,=\,\gamma h_p^{-1}.
\end{equation}
For a weight component \(w\,=\,w_{p,j}\), the order condition and the preceding
formula give
\begin{equation}\label{eq:app-normalized-component-orders}
\begin{array}{@{}cccccc@{}}
\toprule
w&k_{p,j}&\beta_j&\gamma_j&\beta_j^{\mathrm{nor}}
&\gamma_j^{\mathrm{nor}}\\
\midrule
w\,<\,0&0&O(1)&O(z)&O(1)&O(z)\\
w\,=\,0&0&O(1)&O(1)&O(1)&O(1)\\
w\,>\,0&-1&O(z)&O(1)&O(1)&O(z).\\
\bottomrule
\end{array}
\end{equation}
Consequently both expressions in
\eqref{eq:app-Hecke-Higgs-transform} extend as logarithmic bundle maps.
They are the global maps \(\Phi_1\) and \(\Phi_2\) appearing in
\eqref{eq:PU21-matrix-Higgs}.

If
\[
V_{p,0}\,:=\,\bigoplus_{a_{p,j}\,=\,0}L_{p,j}\subset V_p,
\]
then the same table gives
\begin{equation}\label{eq:app-normalized-residue-condition}
\operatorname{Res}_p(\Phi_1)\in V_p,
\qquad
\operatorname{Res}_p(\Phi_2)
\in V_{p,0}^{\vee}\subset V_p^{\vee}.
\end{equation}
Thus the residue of \(\Phi_1\) is unrestricted, whereas the residue of
\(\Phi_2\) vanishes on every positive-weight eigenline.  This is the
\(\mathrm{PU}(2,\,1)\) form of the order condition in
\cite[equation~(4.1)]{BGM2020}.

The resulting normalised representative is recorded by
\[
\bigl(\alpha_p,Q_p;\,\nu_p^{\mathrm{nor}},
L_{\alpha_p}\!\cdot(s_p+Y_p)\bigr).
\]
The Hecke normalisation only selects a preferred vector-bundle lattice
associated with its standard representation.  Once \(\alpha_p\), its
ordered eigenlines, and the target interval \([0,1)\) are fixed,
\(\nu_p^{\mathrm{nor}}\) is determined.  This is the representative used
subsequently when degrees, saturated subbundles, and zero divisors are
computed.

\end{document}